\documentclass[preprint,12pt]{elsarticle}

\usepackage{amssymb}
\usepackage{amsmath}
\usepackage{amsthm}
\usepackage{amssymb}
\usepackage{mathtools}
\usepackage{bbm}
\usepackage{subcaption}
\usepackage{graphicx}
\usepackage[colorlinks=true, allcolors=blue]{hyperref}
\usepackage{url}
\newtheorem{lemma}{Lemma}
\newtheorem{theorem}{Theorem}[section]
\newtheorem{remark}{Remark}
\newtheorem{proposition}{Proposition}
\newtheorem{assumption}{Assumption}
\newtheorem{definition}{Definition}
\allowdisplaybreaks

\begin{document}

\begin{frontmatter}

\title{On Construction, Properties and Simulation of Haar-Based Multifractional Processes}

\author[label1]{Antoine Ayache}
\ead{antoine.ayache@univ-lille.fr}
\author[label2]{{Andriy Olenko}\texorpdfstring{\corref{cor1}}{*}}
\ead{A.Olenko@latrobe.edu.au}
\author[label2]{Nemini Samarakoon}
\ead{n.wijesinghesamarakoon@latrobe.edu.au}
\cortext[cor1]{Corresponding author}
\affiliation[label1]{
            organization={University of Lille},
            addressline={CNRS, UMR 8524 - Laboratoire Paul Painlevé},
            city={Lille},
            postcode={F-59000},
            country={France}}

\affiliation[label2]{
            organization={La Trobe University},
            city={Melbourne},
            country={Australia}}

\begin{abstract}
Multifractional processes extend the concept of fractional Brownian motion by replacing the constant Hurst parameter with a time-varying Hurst function. This extension allows for modulation of the roughness of sample paths over time. The paper introduces a new class of multifractional processes, the Gaussian Haar-based multifractional processes (GHBMP), which is based on the Haar wavelet series representations. The resulting processes cover a significantly broader set of Hurst functions compared to the existing literature, enhancing their suitability for both practical applications and theoretical studies. The theoretical properties of these processes are investigated. Simulation studies conducted for various Hurst functions validate the proposed model and demonstrate its applicability, even for Hurst functions exhibiting discontinuous behaviour.
\end{abstract}




\end{frontmatter}

\section{Introduction}
\label{sec1}

Fractional Brownian motion (fBm) $\{B_{H}(t),\,t \geq 0\}$ was introduced by  \cite{mandelbrot1968fractional}. fBm is a centered, self-similar Gaussian process with stationary dependent increments and it can be defined as follows:
\begin{eqnarray}\label{eq01}
    B_{H}(t) = \dfrac{1}{\tau (H + \frac{1}{2})}&\left\{ \int_{-\infty}^{0} \left[(t-s)^{H-\frac{1}{2}}-(-s)^{H-\frac{1}{2}}\right] dB(s)\right.\nonumber\\ &+\left.\int_{0}^{t} (t-s)^{H-\frac{1}{2}} dB(s) \right \}.
\end{eqnarray}
It has a constant Hurst parameter $H\in (0,1),$  which regulates the roughness of the process.
One of the most frequently used stochastic processes, the standard Brownian motion $\{B(t),\, t \geq 0\},$ is a particular case of fBm with the value of the Hurst parameter equals to ${1}/{2}$. When $H \in (0,{1}/{2}),$ the increments of fBm are negatively correlated and hence it exhibits the short-range dependence (short memory) behaviour. If $H \in ({1}/{2},1)$ the increments are positively correlated and fBm has long-range dependence (long memory). Due to dependent increments, fBm is a more realistic model compared to the standard Brownian motion. Moreover, fBm could model more natural phenomena due to its long memory. fBm has found numerous applications in finance, queuing networks, signal processing, and image processing, see, for example, \cite{BIERME2020293, BROADBRIDGE_2022, decreusefond1998fractional, doukhan2002theory, nualart2006fractional, PIPIRAS200549,  rostek2013note, yancong2011image}, just to mention a few.

However, fBm is unable to model changes in the roughness of trajectories over time. To address the limitations of fBm, several models of multifractional processes $\{B_{H(t)}(t),$ $t \geq 0\},$ have been introduced. This development began with the pioneering works of \cite{benassi1997elliptic}  and \cite{peltier1995multifractional}. The concept of multifractional process proved to be very useful in many applications (see, for example, about finance in \cite{bianchi2011modeling, bianchi2014multifractional,
bianchi2014asset, frezza2018fractal}). The theoretical construction of multifractional processes by Peltier and Véhel was based on the replacement of the constant Hurst parameter in the stochastic integral in (\ref{eq01}) with a deterministic Hurst function $H(t) \in (0,1)$ and was named Classical Multifractional Brownian Motion. The article \cite{benassi1997elliptic} proposed the method to construct multifractional processes through a self-similar Gaussian process focusing on the harmonizable representation of the fBm.

However, further generalizations were uncertain in these cases as introducing a stochastic Hurst parameter can only be achieved independently of the Brownian motion used. Due to these constraints utilizing a random wavelet series was suggested in constructing multifractional processes, where the Hurst function is random.
The first construction of Multifractional Processes with Random Exponent (MPRE) by substituting the Hurst parameter with a stochastic process was suggested by \rm{\cite{ayache2005multifractional}}.  The article \rm{\cite{stoev2006rich}} derived a generalization of the methods introduced in \rm{\cite{benassi1997elliptic}}  and \rm{\cite{peltier1995multifractional}} as a harmonizable representation with stochastic integrals. Later in 2007, the Generalized Multifractional Processes with Random Exponent were introduced, where the Hurst parameter was replaced by a sequence of continuous stochastic processes, see \rm{\cite{ayache2007wavelet}}. Further details regarding the connection of wavelet methods to multifractional processes can be found in \rm{\cite{ayache2018multifractional}}. A general approach in~\cite{lebovits2014stochastic} extends the definitions in \cite{ benassi1997elliptic, peltier1995multifractional, stoev2006rich} while preserving the fundamental nature of this class of Gaussian processes.

Moving Average Multifractional Process with Random Exponent (MAMPRE) was introduced using the It\^o integral with a random Hurst function which depends on the integration variable in~\rm{\cite{ayache2018new}}. MAMPRE replaced the Hurst parameter in~(\ref{eq01}) by the stochastic process. Apart from the It\^o integral representation the MAMPRE processes have advantages in simulations over the MPRE ones, see~\rm{\cite{ayache2018new}}.

Statistical estimation of the Hurst function of multifractional processes is essential for understanding the roughness of these processes in various applications. In this paper, we employ one of such estimates to validate the proposed construction through simulation studies. The majority of known approaches utilize generalized quadratic or power variations methods to build the estimators (\rm{\cite{ayache2022uniformly, ayache2017uniformly, ayache2004identification,bardet2013nonparametric, benassi2000identification, benassi1998identifying, benassi1998identification,  ISTAS1997407,lebovits2017estimation,peng2018general}}). Other more or less related methods include the IR-estimator, which is based on the increment ratio statistic \rm{\cite{bardet2013nonparametric}}, an estimator based on wavelets \rm{\cite{stolojescu2012comparison}}, the rescaled range analysis (R/S) method~\rm{\cite{tzouras2015financial}}, and detrended fluctuation analysis (DFA)~\rm{\cite{tzouras2015financial}}.

It should be noted that multifractional processes do not have stationary increments. There are several methods for constructing them, resulting in different processes even when their H\"older regularities are identical. One of the main challenges is to propose a construction that can provide a feasible representation for theoretical studies and simulations for wide classes of functions $H(t).$ For example, for Hurst functions that can abruptly change over time. In \cite{ayache2018new}, a new approach to multifractional models with random exponents based on the Haar wavelets was suggested.

The main aim of this paper is to utilize the Haar wavelet approach in constructing a new class of multifractional processes. The Haar wavelets were introduced by A.Haar in~1910 and are a system of piecewise constant square wave functions.
Among the wavelet methods, which utilize an analytic expression, the Haar wavelet approach is the simplest and the fastest one as it is based on a system of two piecewise constant functions. Also, numerous standard algorithms exist for the Haar wavelet computations which can achieve the required precision with a minimum number of grid points, see~\rm{\cite{lepik2014haar}}. Due to the above reasons, the Haar wavelet approach is used in this paper to construct a new class of multifractional processes, which will be referred to as Gaussian Haar-based multifractional processes (GHBMP).

It is worth noting that the proposed representation is straightforward to implement in program code and practical to use. However, rigorously verifying that a specific stochastic series possesses the required properties remains a highly challenging mathematical problem, as evidenced by the publications cited above. This paper extends this line of research by developing a methodology to establish results regarding the properties of random series with what is likely the most general class of H\"older exponents studied to date.

The paper is organized as follows. Section~\ref{sec2} introduces the considered class of multifractional processes. The next section proves the correctness of their definition and some basic properties. Sections~\ref{sec3} and~\ref{sec4} investigate pointwise H\"older exponents of these multifractional processes. Section~\ref{sec5} provides the simulation studies performed for different Hurst functions. To confirm the theoretical results, in this section the Hurst functions were estimated from the simulated multifractional process and compared with the theoretical functions. The last section discusses the obtained results and identifies some problems for future research.

\section{Preliminaries}\label{sec2}

This section introduces the main definitions and the class of multifractional processes that will be studied in this paper. Some properties of these processes are investigated.

$\mathbb{Z}_{+}$ denotes the set of all non-negative integer numbers.  $C^\infty$ is the space of all infinitely differentiable functions. The notation~$L_2$ will be used for the space of square-integrable functions and the corresponding norm. The standard deviation of the random variable $X$ will be denoted by $\sigma(X).$ We also use the convention that the sum over an empty set of indices equals zero.

In the following, $C$ and $a$ with subindices represent generic finite positive constants. The values of such constants are not necessarily the same in each appearance and may change depending on the expression.

All stochastic processes considered in this article are defined on the same probability space denoted by $(\Omega, \cal{F}, {\mathbb P}).$ For a process $\{X(t),$ $t\in [0,1]\}$ with continuous sample paths, the pointwise H\"older exponent defined below is used to quantify the roughness of the process in a neighbourhood of any chosen fixed time point.

\begin{definition}\label{def1}
   The H\"older exponent of the stochastic process $X(\cdot)$ at point $t\in[0,1]$ is defined~by
\[\alpha_X(t) = \sup \left\{ \alpha \in \mathbb{R}_+ : \limsup_{h \to 0} \frac{|X(t + h) - X(t)|}{|h|^{\alpha}} = 0 \right\}.
\]
\end{definition}
For each fixed $t\in[0,1]$ the H\"older exponent $\alpha_X(t)$ is a random variable.

Let $\{H_{j}:j \in \mathbb{Z}_{+}\}$ be a sequence of Lipchitz functions defined on $[0,1]$ with values in a fixed compact interval $[\underline{H},\overline{H}] \subset (0,1).$

\begin{definition} For each $j \in
\mathbb{Z}_{+},$ we define the norm of $H_{j}$ by
\begin{equation} \label{eq1}
    \nu_{j} \coloneqq \left\Vert H_{j} \right\Vert_{Lipz} \coloneqq \sup_{x \in [0,1]} |H_{j}(x)| + \sup_{0 \leq x^\prime < x^{\prime\prime}} \dfrac{|H_{j}(x^{\prime\prime})-H_{j}(x^\prime)|}{x^{\prime\prime}-x^\prime}.
    \end{equation}
    \end{definition}
\begin{assumption}
For the sequence of Lipchitz functions $\{H_{j}:j \in \mathbb{Z}_{+}\}$  there exists a constant  $C>0$ such that, for all $j \in \mathbb{Z}_{+},$ it holds
\begin{equation} \label{eq2}
    \nu_{j} \leq C(1+j).
\end{equation}
\end{assumption}

\begin{definition}
The Haar mother wavelet $h$ is defined via the indicator function $\mathbbm{1}_A(\cdot)$ of a set $A,$  as
    \begin{equation} \label{eq3}
        h(s) :=  {\mathbbm{1}}_{[0,1/2
        )}(s) -{\mathbbm{1}}_{[1/2,1       )}(s), \quad s \in \mathbb{R}.
    \end{equation}
For all $j \in \mathbb{Z}_{+},$  $k \in \{0, \dots , 2^{j}-1\}$ and $s \in \mathbb{R}$ their dyadic translations and dilations are given by
\begin{equation} \label{eq4}
    h_{j,k} (s) \coloneqq 2^{j/2} h(2^{j}s-k)
    = 2^{j/2} \left({\mathbbm{1}}_{[k/2^j,(k+1/2)/2^j)}(s) -{\mathbbm{1}}_{[(k+1/2)/2^j,(k+1)/2^j   )}(s)\right).
\end{equation}
\end{definition}

Let $ \{ \varepsilon_{j,k}:j \in \mathbb{Z}_{+}$ and $k \in \{ 0, \dots , 2^{j}-1\} \}$ be a sequence of independent $\mathcal{N}$(0,1) Gaussian random variables.

To model the multifractal behaviour this paper introduces and investigates the following class of processes.

\begin{definition}\label{def_pro} The process defined on the interval $[0,1]$ as
\begin{equation}\label{eq6}
    X(t) \coloneqq \sum_{j=0}^{+\infty}  \sum_{k=0}^{2^{j}-1}
     \left(\int_{0}^{1} (t-s)_{+}^{H_{j}(k/{2^j})-{1}/{2}} h_{j,k}(s)ds \right)\varepsilon_{j,k},
\end{equation}
 will be called the Gaussian Haar-based multifractional process. The abbreviation GHBMP will be used to denote processes from this class.
\end{definition}

 Note that the GHBMP processes defined by (\rm {\ref{eq6}}) are centered Gaussian, but do not have stationary increments.

\section{On GHBMP and its H\"older continuity}

This section studies GHBMP and its basic properties. It also provides several key characteristics of Haar-based transforms that will be used in subsequent analyses.

The first result shows that the series in Definition~{\rm \ref{def_pro}} is $L_2$-convergent and GHBMP is correctly defined.

\begin{theorem}\label{theorem1}
    For all fixed $t \in [0,1]$ it holds
    \begin{equation*}
       {\mathbb E}X^2(t)=  \sum_{j=0}^{+\infty}  \sum_{k=0}^{2^{j}-1}
     \left|\int_{0}^{1} (t-s)_{+}^{H_{j}(k/{2^j})-{1}/{2}} h_{j,k}(s)ds\right|^2 < +\infty.
    \end{equation*}
 Therefore, the GHBMP processes are properly defined in $L_2$-sense.
\end{theorem}

The proof of Theorem~{\rm \ref{theorem1}} relies on the following result.

\begin{lemma} \label{lemma1}
    For any fixed value $\lambda \in [\underline{H},\overline{H}],$ let $h^{[\lambda]}(\cdot)$ be the function defined on $\mathbb{R}$ as
    \begin{align} \label{eq8}
        h^{[\lambda]} (x) &: =  \int_{\mathbb{R}} (x-s)_{+}^{\lambda-\frac{1}{2}} h(s) ds \nonumber\\
        & =  \left(\lambda + \frac{1}{2}\right)^{-1}
        \left( (x)_+^{\lambda +\frac{1}{2}} -2\left(x-\frac{1}{2}\right)_+^{\lambda +\frac{1}{2}} + (x-1)_+^{\lambda +\frac{1}{2}} \right).
     \end{align}

     Then, there is a constant $c>0$ such that for all $(\lambda,x) \in [\underline{H},\overline{H}] \times \mathbb{R},$

     \begin{equation} \label{eq9}
         \left|h^{[\lambda]} (x)\right| \leq c\left(3+|x|\right)^{\lambda-\frac{3}{2}}.
     \end{equation}
\end{lemma}

For any $(y,\beta) \in \mathbb{R}^2,$ we use the next notation
\begin{equation*}
    (y)_+ ^ \beta : =
\begin{cases}
y^ \beta ,& \mbox{when}\ y\ge 0,\\
0, & \text{else.}
\end{cases}
\end{equation*}

\begin{remark}
    For each $\beta \in (-1,+\infty),$ and for all real numbers $a<b,$ it holds
    \begin{equation*}
        \int_{a}^{b} (y)_{+}^{\beta} dy = (\beta +1)^{-1} \left((b)_{+}^{\beta +1}-(a)_{+}^{\beta+1}\right).
    \end{equation*}
\end{remark}

\begin{proof}[Proof of Lemma~{\rm \ref{lemma1}}]
It follows from ({\rm \ref{eq3}}) that
\begin{align}\label{hl}
    h^{[\lambda]} (x) &= \int_{0}^{\frac{1}{2}} \left( \left(x-s\right)^{\lambda - \frac{1}{2}}_+ - \left(x-\frac{1}{2}-s\right)^{\lambda - \frac{1}{2}}_+ \right) ds \\
    & = \left(\lambda + \frac{1}{2}\right)^{-1} \left( (x)_+^{\lambda + \frac{1}{2}} -2\left(x-\frac{1}{2}\right)^{\lambda + \frac{1}{2}}_+ + \left(x-1\right)^{\lambda + \frac{1}{2}}_+\right).\nonumber
\end{align}

Therefore,
       \begin{equation} \label{eq10}
h^{[\lambda]} (x) =0, \ \mbox{if}\ x \in (-\infty,0],
     \end{equation}
and  (\ref{eq9}) holds on the interval $(-\infty,0].$

Note that for each finite interval, there exists a constant $c,$ such that (\ref{eq9}) is valid. For example, it follows from (\ref{eq8}) that
    \begin{equation} \label{eq11}
        |h^{[\lambda]} (x)| \leq  4 \left(\lambda + \frac{1}{2}\right)^{-1} 4^{\lambda + \frac{1}{2}}\leq 8\cdot 4^\frac{3}{2} \leq c (3+4)^{-\frac{3}{2}}\leq c\left(3+|x|\right)^{\lambda-\frac{3}{2}},
    \end{equation}
    when $x \in [0,4],$

For the case of $x \in [4, +\infty),$ by (\ref{hl}), monotonicity of the integrands and the mean value theorem, one obtains that
    \begin{align*}
        \left|h^{[\lambda]} (x)\right| &\leq \frac{1}{2} \left| x^{\lambda - \frac{1}{2}} - \left(x-1\right)^{\lambda - \frac{1}{2}}\right|  = \frac{1}{2}\left|\lambda - \frac{1}{2}\right| (x-a)^{\lambda-\frac{3}{2}}\leq  \frac{1}{4} (x-1)^{\lambda-\frac{3}{2}},
    \end{align*}
    where $a\in (0,1).$

    Then, as  $\lambda \in [\underline{H},\overline{H}] \subset (0,1),$ it follows that
    \begin{align} \label{eq13}
        \left|h^{[\lambda]} (x)\right| &\leq \frac{1}{4} \left(x-1\right)^{\lambda - \frac{3}{2}} = \frac{1}{4} \left(\dfrac{x-1}{x+3}\right)^{\lambda - \frac{3}{2}}  \left(3+x\right)^{\lambda - \frac{3}{2}} \nonumber\\
        &\leq  \frac{1}{4} \left(1+\dfrac{4}{x-1}\right)^\frac{3}{2} \left(3+x\right)^{\lambda - \frac{3}{2}}
        \leq \frac{1}{4} \left(1+\frac{4}{3}\right)^\frac{3}{2} \left(3+x\right)^{\lambda - \frac{3}{2}}.
    \end{align}

    Finally, by combining ({\rm \ref{eq10}}), ({\rm \ref{eq11}}) and ({\rm \ref{eq13}}), one obtains the statement (\ref{eq9}) of the lemma.
\end{proof}

\begin{remark}\label{remark1}
    For the sake of simplicity, we will denote $H_{j,k}: = H_j\left(\frac{k}{2^j}\right).$
\end{remark}

\begin{lemma} \label{lemma2}
    For all $j \in \mathbb{Z}_{+},$ $k \in \{0, \dots, 2^{j}-1\}$ and $t \in [0,1],$ it holds that
    \begin{equation*}
        \int_{0}^{1} (t-s)_{+}^{H_{j,k}-\frac{1}{2}} h_{j,k} (s) ds = 2^{-j H_{j,k}} h^{[H_{j,k}]} (2^jt-k).
    \end{equation*}
\end{lemma}

\begin{proof}[Proof of Lemma~{\rm \ref{lemma2}}]
    Using (\ref{eq4}) and the change of variable $s' = 2^{j}s -k,$ one gets the next relation
    \begin{align}\label{eq14}
        \int_{0}^{1} (t-s)_{+}^{H_{j,k}-\frac{1}{2}} h_{j,k}(s) ds &= \int_{\mathbb{R}} (t-s)_{+}^{H_{j,k}-\frac{1}{2}} 2^{\frac{j}{2}} h(2^{j}s -k) ds \\
        & = 2^{-\frac{j}{2}} \int_{\mathbb{R}} \left(t-(2^{-j}k+2^{-j}s')\right)_{+}^{H_{j,k}-\frac{1}{2}} h(s')ds'\nonumber \\
        & = 2^{-jH_{j,k}} h^{[H_{j,k}]} (2^{j}t-k), \nonumber
    \end{align}
where the last equality follows from (\ref{eq8}).
\end{proof}

\begin{proof}[Proof of Theorem~{\rm \ref{theorem1}}]
    By Remark~\ref{remark1}, Lemmas~\ref{lemma1} and~\ref{lemma2}, and  the fact that the $H_{j,k}$ belong to $[\underline{H},\overline{H}],$  one has
    \begin{align*}
        &\sum_{j=0}^{+\infty} \sum_{k=0}^{2^{j}-1} \left|\int_{0}^{1} (t-s)_{+}^{H_{j,k}-\frac{1}{2}} h_{j,k}(s) ds\right|^2
        \leq \sum_{j=0}^{+\infty} 2^{-2j\underline{H}} \sum_{k=0}^{2^{j}-1} \left|h^{[H_{j,k}]}(2^jt-k)\right|^2 \\
        &\quad \leq \sum_{j=0}^{+\infty} 2^{-2j\underline{H}} \sum_{k=0}^{2^{j}-1} C_1\left(3+|2^jt-k|\right)^{2\overline{H}-3}  \leq \dfrac{C_1C_2}{1-2^{-2\underline{H}}} < +\infty,
    \end{align*}
    where $C_2 \coloneqq \sum_{p=-\infty}^{+\infty} \left(2+|p|\right)^{2\overline{H} -3} < \infty.$
\end{proof}
The following result is a direct consequence of the representation {\rm(\ref{eq6})} and  Lemma~\ref{lemma2}.
\begin{proposition}
  The covariance function of the GHBMP process defined by~{\rm(\ref{eq6})}   equals
  \[{\rm Cov}(X(t),X(t'))=\sum_{j=0}^{+\infty}  \sum_{k=0}^{2^{j}-1} 2^{-2jH_{j,k}} h^{[H_{j,k}]} (2^{j}t-k) h^{[H_{j,k}]} (2^{j}t'-k).\]
\end{proposition}

\begin{theorem}\label{prop5}
    For all $t',t'' \in [0,1]$ it holds
    \begin{equation*}
        {\mathbb E}(|X(t'')-X(t')|^2)\leq c|t''-t'|^{2\underline{H}}.
    \end{equation*}
    Therefore, the paths of $X$ are H\"older continuous functions of any order strictly lower than $\underline{H}.$
\end{theorem}

\begin{proof}[Proof of Theorem~{\rm \ref{prop5}}]
     Without loss of generality, one can assume that $0<t''-t'\leq 1.$ Let $j_1$ be the unique non-negative integer such that
    \begin{equation}\label{eq17}
        2^{-j_1-1}<t''-t'\leq2^{-j_1}.
    \end{equation}
 By the mean value theorem, using the fact that  $h^{[H_{j,k}]} \in C^\infty([2,+\infty]),$    one obtains that for all $j \in \{0, \dots, j_1\}$ and $k\leq [2^jt']-2$ it holds
    \begin{equation}\label{eq18f}
        \left|h^{[H_{j,k}]}(2^jt''-k)-h^{[H_{j,k}]}(2^jt'-k)\right| \leq 2^j(t''-t') \left(h^{[H_{j,k}]}\right)'\left(a_1-k\right),
    \end{equation}
    where $a_1 \in (2^jt',2^jt'').$ Moreover, by applying the mean value theorem two times, one obtains that for some constant $C_1$ not depending on $t',$ $t'',$  $j_0,$ $j$ and $k$ it holds
    \begin{equation}\label{eq18}
        \left|\left(h^{[H_{j,k}]}\right)'\left(a_1-k\right)\right|\leq C_1\left(2^jt'-k-1\right)^{H_{j,k}-\frac{5}{2}} \leq C_1\left(2^jt'-k-1\right)^{\overline{H}-\frac{5}{2}}.
    \end{equation}

Combining (\ref{eq17}) and (\ref{eq18}), one gets
\begin{align}
    &\sum_{j=0}^{j_{1}} \sum_{k=0}^{[2^jt']-2} 2^{-2jH_{j,k}} \left|h^{[H_{j,k}]}(2^jt''-k)-h^{[H_{j,k}]}(2^jt'-k)\right|^2 \nonumber \\
    &\quad\leq C_1^2 \sum_{j=0}^{j_1} 2^{2j(1-\underline{H})} \sum_{k=0}^{[2^jt']-2} \left(2^jt'-k-1\right)^{2\overline{H}-5} (t''-t')^2 \nonumber \\
    &\quad\leq C_1^2 \left(\sum_{p=1}^{+\infty} p^{2\overline{H}-5}\right) 2^{2(j_1+1)(1-\underline{H})} (t''-t')^2 \leq C_2(t''-t')^{2\underline{H}}.\label{eq19}
\end{align}

Now we investigate the case of $j \in \{0, \dots, j_{1}\}$ and $k \in \{[2^jt']-1,\dots,[2^jt'']\}.$ Notice that in this case the latter set for $k$ consists of at most 3 elements.

We will use the inequality
\begin{equation}\label{eq20}
    |(b)_{+}^{\beta}-(a)_{+}^{\beta}| \leq 2(1+\kappa) |b-a|^{\beta\wedge 1},
\end{equation}
which holds for all $\beta \in (0,2)$ and real numbers $a,b \in [-\kappa,\kappa]$, where $\kappa>0$ is an arbitrary fixed constant.

By applying (\ref{eq8}) and (\ref{eq20}) with $\kappa=4$, one gets

\begin{align}
&\left|h^{[H_{j,k}]}(2^jt''-k)-h^{[H_{j,k}]}(2^jt'-k)\right| \nonumber\\
    &\quad\leq 4 \sum_{l=0}^{2} \left|\left(2^jt''-k-\frac{l}{2}\right)_{+}^{H_{j,k}+\frac{1}{2}}-\left(2^jt'-k-\frac{l}{2}\right)_{+}^{H_{j,k}+\frac{1}{2}}\right|\nonumber \\
    &\quad\leq 40\cdot \left|2^j(t''-t')\right|^{\left(H_{j,k}+\frac{1}{2}\right)\wedge 1},\label{incr}
\end{align}

and consequently, by (\ref{eq17}),
\begin{align}\label{eq21}
    \sum_{j=0}^{j_{1}}& \sum_{\substack{k=[2^jt']-1 \\ k\geq0}}^{[2^j t'']} 2^{-2jH_{j,k}} \left|h^{[H_{j,k}]}(2^jt''-k)-h^{[H_{j,k}]}(2^jt'-k)\right|^2 \nonumber\\
    & \leq C  \sum_{j=0}^{j_1} 2^{j\left((2H_{j,k}+1)\wedge 2-2H_{j,k}\right)} |t''-t'|^{(2H_{j,k}+1)\wedge 2} \nonumber\\
    &  \leq C\cdot \sum_{j=0}^{j_1}2^{j\left(1\wedge2(1-\underline{H})\right)} \left|t''-t'\right|^{(2\underline{H} +1 )\wedge 2}\leq C\cdot 2^{j_1\left(1\wedge2(1-\underline{H})\right)} \left|t''-t'\right|^{(2\underline{H} +1 )\wedge 2}\nonumber\\
    & \leq C\cdot  |t''-t'|^{-(1\wedge2(1-\underline{H}))+(2\underline{H}+1)\wedge 2}  =C\cdot  |t''-t'|^{2\underline{H}},
\end{align}
since
    \begin{align*}
        &-\left(1\wedge 2(1-\underline{H})\right) + (2\underline{H}+1)\wedge 2 = \dfrac{-\left(1+2(1-\underline{H})-|2(1-\underline{H})-1|\right)}{2} \\
         &\quad + \dfrac{2\underline{H}+1+2-|2\underline{H}+1-2|}{2}
         = 2\underline{H}.
    \end{align*}
Now, let us consider the case of $j \geq j_1+1.$
Using (\ref{eq9}) and (\ref{eq17}), one obtains that for all $t \in [0,1]$ it holds
\begin{align} \label{eq22}
    &\sum_{j=j_1+1}^{+\infty} \sum_{k=0}^{2^j-1} 2^{-2j H_{j,k}} \left|h^{[H_{j,k}]}(2^jt-k)\right|^2  \leq C\cdot\sum_{j=j_1+1}^{+\infty} 2^{-2j\underline{H}} \sum_{k \in \mathbb{Z}} \left(3+|2^jt-k|\right)^{2\overline{H}-3} \nonumber \\
    &\quad \leq C\cdot\sum_{p \in \mathbb{Z}} (2+|p|)^{2\overline{H}-3} \cdot  2^{-2(j_1+1)\underline{H}} \leq C\cdot\left|t''-t'\right|^{2\underline{H}}.  \end{align}
 It follows from (\ref{eq22}) and the inequality $(a-b)^2\leq 2a^2 + 2b^2$ that
 \begin{equation} \label{eq23}
     \sum_{j=j_1+1}^{+\infty} \sum_{k=0}^{2^j-1} 2^{-2j H_{j,k}} \left|h^{[H_{j,k}]}(2^jt''-k)-h^{[H_{j,k}]}(2^jt'-k)\right|^2 \leq C\cdot\left|t''-t'\right|^{2\underline{H}}.
 \end{equation}

Combining the results in (\ref{eq19}), (\ref{eq21}) and (\ref{eq23}) we obtain the statement of Theorem~{\rm \ref{prop5}}.
\end{proof}

\section{Lower bound for H\"older exponent}\label{sec3}
This section derives a lower bound for the  H\"older exponent $\alpha_X(t).$

\begin{assumption} \label{ass1}   For all $t \in (0,1)$ it holds that
    \begin{equation}\label{eq23'}
        H(t)\coloneqq \liminf_{j \rightarrow +\infty} H_{j}(t) < \underline{H} + {1}/{2}.
    \end{equation}
\end{assumption}

For all integers $J \in \mathbb{Z}_{+}$ and $K \in \{0, \dots, 2^{J}-1\}$, we set $d_{J,K}: = {K}/{2^J}$ and  use the notation
\begin{equation}\label{eq25}
    \Delta_{J,K} :=  X(d_{J,K+1}) - 2X(d_{J+1,2K+1}) + X(d_{J,K}).
\end{equation}
Our next objective is to prove the following result.

\begin{proposition}\label{prop1}
There is a universal event of probability $1,$ denoted by $\Omega_{1}^{*},$ such that on $\Omega_{1}^{*}$,  it holds
\[\alpha_{X}(t)\geq H(t)\quad \mbox{for all}\quad t\in (0,1).\]
\end{proposition}

\begin{proof}[Proof of Proposition~{\rm\ref{prop1}}] The key approach of the proof is based on applying Proposition~4 from  {\rm\cite{daoudi1998construction}}. We divide the proof into several lemmas.
\begin{lemma}\label{lemma 9}
There exists an event $\Omega_{1}^{*}$ of probability $1$  and  a random variable $C^{*}_{1}$ with finite moments of any order, such that the inequality
\begin{equation*}
    |\Delta_{J,K}| \leq C^{*}_{1} \sigma (\Delta_{J,K}) \sqrt{\log(3+J+K)}
\end{equation*}
 holds on $\Omega_{1}^{*}$ for all $J\in \mathbb{Z}_{+}$ and $K \in \{0, \dots ,2^{J}-1 \}.$
\end{lemma}

\begin{proof}[Proof of Lemma~{\rm \ref{lemma 9}}]
We will use the fact that, for $Z\sim N(0,1),$ the tail probability $ {\mathbb P}(Z>z)$ can be estimated as follows
\begin{equation}\label{normtail}
 {\mathbb P}(Z>z)\le \frac{\exp(-z^2/2)}{\sqrt{2\pi}z}.\end{equation}
 As  $\{\Delta_{J,K}\}_{J,K}$ are centered Gaussian random variables, then, for $c_{0}>0,$
    \begin{align*}
        & {\mathbb E} \left ( \sum_{J=0}^{+\infty} \sum_{K=0}^{2^J-1} \mathbbm{1}_{\{
|\Delta_{J,K}| > c_{0} \sigma (\Delta_{J,K}) \sqrt{\log(3+J+K)} \}}\right )
 =  \sum_{J=0}^{+\infty} \sum_{K=0}^{2^J-1} {\mathbb P}\left(
|Z| > c_{0} \sqrt{\log(3+J+K)} \right )\\
& < \sum_{J=0}^{+\infty} \sum_{K=0}^{2^J-1} \frac{2}{c_0(3+J+K)^{c_0^2/2}\sqrt{2\pi\log(3+J+K)}} <C\int_0^{\infty}\int_0^\infty (3+x+y)^{-c_0^2/2}dxdy.
    \end{align*}
   Selecting, say, $c_0>\sqrt{6}$ guarantees that the above upper bound is finite, which  implies that
    \begin{equation*}
        C^{*}_{1} \coloneqq \sup_{J \in \mathbb{Z}_{+}} \sup_{0 \leq K< 2^{J}} \dfrac{|\Delta_{J,K}|}{\sigma (\Delta_{J,K}) \sqrt{\log(3+J+K)}}
    \end{equation*}
    is an almost surely finite random variable.

    Then, by applying the Borell-TIS inequality,  we get ${\mathbb E}(|C_{1}^{*}|^{p})< +\infty$ for all~\mbox{$p>0.$}
\end{proof}

Now, for a fixed $t_0\in (0,1),$ let us consider arbitrary $J\in \mathbb{Z}_{+}$ and $K \in \{0, \dots , 2^J -1 \}$ satisfying
    \begin{equation}\label{eq27}
        \left|\dfrac{K}{2^{J}}-t_{0}\right| \leq (1+J)^{-4}.
    \end{equation}
Let us assume that $a_0>2$ is a fixed constant such that for all $J \in \mathbb{Z_{+}}$
\begin{equation} \label{eq27b}
    2^{-1}a_{0} (1+J)^{-4} > 2^{-J}.
\end{equation}

 For all $j \in \mathbb{Z_{+}}$, the finite sets $\nu_{J,j}^{K}$ and $\bar{\nu}_{J,j}^{K}$ are defined as
\begin{equation}\label{eq30}
    \nu_{J,j}^{K} \coloneqq \left\{ k \in \{ 0, \dots , 2^{j}-1\} : |d_{j,k}-d_{J,K}|\leq a_{0} (1+J)^{-4} \right\},
\end{equation}
and
\begin{equation}\label{eq31}
    \bar{\nu}_{J,j}^{K} \coloneqq \left\{ k \in \{ 0, \dots , 2^{j}-1\} : |d_{j,k}-d_{J,K}|> a_{0} (1+J)^{-4} \right\}.
\end{equation}
For every $j \in \mathbb{Z_{+}}$ and $k \in \{0,\dots , 2^{j}-1\}$, one sets
\begin{equation}\label{eq32}
    D_{J,j}^{K,k} \coloneqq h^{[H_{j,k}]} (2^{j}d_{J,K+1}-k)-2h^{[H_{j,k}]} (2^jd_{J+1,2K+1}-k) + h^{[H_{j,k}]} (2^{j}d_{J,K}-k).
\end{equation}
Then, by  ({\rm \ref{eq6}}), Lemma~\ref{lemma2} and  ({\rm \ref{eq25}}), $\Delta_{J,K}$ can be expressed as
\begin{equation*}
    \Delta_{J,K} = \sum_{j=0}^{+\infty} \sum_{k=0}^{2^{j}-1} 2^{-jH_{j,k}} D_{J,j}^{K,k} \varepsilon_{j,k}.
\end{equation*}
Therefore,
\begin{equation*}
    {\mathbb E}(|\Delta_{J,K}|^2) = M_{J}^{K} + R_{J}^{K},
\end{equation*}
where
\begin{equation}\label{M+R}
    M_{J}^{K} \coloneqq \sum_{j=0}^{+\infty} \sum_{k\in \nu_{J,j}^{K}} 2^{-2jH_{j,k}} |D_{J,j}^{K,k}|^{2},
\end{equation}
and
\begin{equation}\label{eq35}
    R_{J}^{K} \coloneqq \sum_{j=0}^{+\infty} \sum_{k\in \bar{\nu}_{J,j}^{K}} 2^{-2jH_{j,k}} |D_{J,j}^{K,k}|^{2}.
\end{equation}
First, let us investigate $R_{J}^{K}.$

\begin{lemma}\label{lemma5}
        For any fixed $t_{0}\in (0,1)$, there is a constant $C>0$, such that, for every $J \in \mathbb{Z}_{+}$ and $K \in \{0, \dots, 2^J -1\}$ satisfying {\rm(\ref{eq27})}, the following inequality holds
                \begin{equation*}
            R_{J}^{K} \leq C 2^{-J((2\underline{H}+1)\wedge 2)}(1+J)^{4}.
        \end{equation*}
    \end{lemma}

 To prove Lemma~\ref{lemma5} we need the following result.

\begin{lemma}\label{lemma4}
   Let $a_0>2$ be the same constant as in {\rm (\ref{eq27b})}. Let $J \in \mathbb{Z}_{+}$ and $K \in \{0, \dots, 2^J -1\}$ be arbitrary and such that {\rm(\ref{eq27})} holds. For all $j \in \mathbb{Z_{+}}$ and $k\in \{0,\dots, 2^{j}-1\}$ such that
    \begin{equation}\label{eq36}
        d_{j,k} > d_{J,K} + a_{0}(1+J)^{-4},
    \end{equation}
    one has $D_{J,j}^{K,k}=0$.
\end{lemma}
\begin{proof}[Proof of Lemma~{\rm \ref{lemma4}}]
    It follows from ({\rm \ref{eq36}}), the definition of $d_{J,K}$ and ({\rm \ref{eq27b}}) that $d_{j,k}>d_{J,K+1}+2^{-1}a_{0}(1+J)^{-4}$,
    which implies that
    \begin{equation*}
        2^{j}d_{j,k} = k > 2^{j} d_{J,K+1},
    \end{equation*}
    i.e. $2^{j} d_{J,K+1}-k<0.$ Then, by the inequalities $2^{j}d_{J,K+1}>2^{j}d_{J+1,2K+1}>2^{j}d_{J,K}$ and the relations ({\rm \ref{eq32}}) and ({\rm \ref{eq8}}), one obtains that $D_{J,j}^{K,k}=0$.
\end{proof}

\begin{proof}[Proof of Lemma~{\rm \ref{lemma5}}]
For  $j \in \mathbb{Z}_{+}$, let us define
    \begin{equation}\label{eq37}
    \tilde{k} (j,J,K) : = \max\{ k:\, k \in {0,\dots, 2^{j}-1},\, d_{j,k}<d_{J,K}-a_{0}(1+J)^{-4} \}
    \end{equation}
    with the convention that $\tilde{k}(j,J,K):=-1,$ when the set over which the maximum is taken is empty.

First, let us consider the case when
\begin{equation}\label{eq38}
    j>\tilde{L}(J)\coloneqq 4 \log_{2}(1+J).
\end{equation}
Then, as $a_{0}>2,$ for  all $k\leq \tilde{k}(j,J,K),$ it holds
\begin{equation}\label{eq39}
    d_{j,k+1}=d_{j,k}+{2^{-j}}<d_{J,K}-(a_{0} (1+J)^{-4}-2^{-j})<d_{J,K}-2^{-1} a_{0} (1+J)^{-4}.
\end{equation}

Setting the values $\beta_0=\beta_2=1$ and $\beta_1=-2,$ it follows from (\ref{eq4}), (\ref{eq32}) and Lemma~\ref{lemma2}  that
\begin{align*}
    2^{-jH_{j,k}} D_{J,j}^{K,k} & = 2^{\frac{j}{2}} \int_{\frac{k}{2^j}}^{\frac{k+{1}/{2}}{2^j}} \sum_{m=0}^{2} \beta_{m} \left\{( d_{J+1,2K+m}-s)^{H_{j,k}-\frac{1}{2}}\right.\\
    & \left.\ \ \ - (d_{J+1,2K+m}-s-2^{-j-1})^{H_{j,k}-\frac{1}{2}} \right\} ds.
\end{align*}
Using the mean value theorem three times and noting that $d_{J,K}-s-2^{-j-1}\in(0,1),$ we obtain
\begin{align}
    2^{-j H_{j,k}} |D_{J,j}^{K,k}| &\leq 2^{-\frac{j}{2}-2J} \int_{d_{j+1,2k}}^{d_{j+1,2k+1}} \left( d_{J,K}-s-2^{-j-1}\right)^{H_{j,k}-\frac{7}{2}} ds\nonumber \\
    & \leq 2^{-\frac{j}{2}-2J} \int_{d_{j+1,2k}}^{d_{j+1,2k+1}} \left( d_{J,K}-s-2^{-j-1}\right)^{\underline{H}-\frac{7}{2}} ds.\label{upb}
\end{align}

Applying the Cauchy-Schwarz inequality to the integrals in (\ref{upb}), we deduce from~(\ref{eq39}) that
\begin{align*}
    \sum_{k=0}^{\tilde{k}(j,J,K)} |2^{-jH_{j,k}}D_{J,j}^{K,k}|^2 & \leq 2^{-2j-4J} \int_{0}^{d_{j,\tilde{k}(j,J,K)+1}} (d_{J,K}-s)^{2\underline{H}-7} ds \\
    & \leq 2^{-2j-4J}\left( d_{J,K} - d_{j,\tilde{k}(j,J,K)+1}\right)^{2\underline{H}-6}\\
    & \leq C 2^{-2j-4J} (1+J)^{4(6-2\underline{H})}.
\end{align*}

Thus, the following upper bound holds true
\begin{equation}\label{eq40}
    \sum_{j=\tilde{L}(J)+1}^{+\infty} \sum_{k=0}^{\tilde{k}(j,J,K)} |2^{-jH_{j,k}}D_{J,j}^{K,k}|^{2} \leq C 2^{-4J} (1+J)^{4(4-2\underline{H})},
\end{equation}

where the constant $C$ does not depend on $J.$

Secondly, let $j$ do not satisfy ({\rm \ref{eq38}}), i.e. $j\le\tilde{L}(J).$ Note that  it follows  from (\ref{eq37}) that $d_{j,k+1}<d_{J,K}$ if $k=0,\dots, \tilde{k}(j,J,K)-1,$ and all estimates from the first case remain unchanged. Therefore, similarly to  ({\rm \ref{eq40}}),  one obtains that
\begin{equation}\label{eq41}
    \sum_{j=0}^{\tilde{L}(J)} \sum_{k=0}^{\tilde{k}(j,J,K)-1}|2^{-jH_{j,k}}D_{J,j}^{K,k}|^{2}\leq C 2^{-4J} (1+J)^{4(6-2\underline{H})}.
\end{equation}
Moreover, by (\ref{eq8}), (\ref{eq20}) (with a fixed constant $\kappa$ large enough), ({\rm \ref{eq32}}) and (\ref{eq38})  we get that
\begin{align}\label{eq42}
    \sum_{j=0}^{\tilde{L}(J)} |2^{-jH_{j,\tilde{k }(j,J,K)}}D_{J,j}^{K,\tilde{k}(j,J,K)}|^{2} & \leq C 2^{-J((2\underline{H}+1)\wedge 2)} \sum_{j=0}^{\tilde{L}(J)} 2^{j(1 \wedge 2(1-\underline{H}))} \nonumber\\
    & \leq C 2^{-J((2\underline{H}+1)\wedge 2)} \cdot 2^{\tilde{L}(J)(1 \wedge 2(1-\underline{H}))} \nonumber\\
    & \leq C 2^{-J((2\underline{H}+1)\wedge 2)} (1+J)^4,
    \end{align}
 where the constant $C$ does not depend on $J.$

Combining Lemma~{\rm \ref{lemma4}}, ({\rm \ref{eq31}}), ({\rm \ref{eq35}}), ({\rm \ref{eq37}}),   ({\rm \ref{eq40}}), ({\rm \ref{eq41}}) and ({\rm \ref{eq42}}) we obtain the statement of Lemma~\ref{lemma5}.
\end{proof}

Let us now establish an upper bound for $M_{J}^{K}$ such that the upper bound for $R_{J}^{K}$ presented in Lemma~\ref{lemma5} becomes negligible in comparison.

    \begin{lemma}\label{lemma6}
        Under the assumption {\rm(\ref{eq23'})}, for any fixed $t_{0}\in (0,1)$ and sufficiently small  $\varepsilon>0$, there is a constant $C>0$, such that, for every $J \in \mathbb{Z}_{+}$ and $K \in \{0, \dots, 2^J -1\}$ satisfying {\rm(\ref{eq27})}, the following inequality holds
        \begin{equation*}
            M_{J}^K\leq C 2^{-2J(H(t_{0})-\varepsilon)}.
        \end{equation*}
    \end{lemma}

    \begin{proof}[Proof of Lemma~{\rm \ref{lemma6}}]
     First notice that for every $t_{0}\in (0,1)$ and~$\varepsilon>0,$ there exists $j_{0} \in \mathbb{Z}_{+}$, that only depends on $t_{0}$ and~$\varepsilon$, such that, for all $j\geq j_{0}$
     \begin{equation}\label{eq45}
         H_{j}(t_{0}) \geq H(t_{0})-\varepsilon.
     \end{equation}
Using (\ref{eq8}), (\ref{eq20}) (with a fixed constant $\kappa$ large enough), (\ref{eq32}), and noting that the following sum has a finite number of terms, one obtains
     \begin{align}\label{uptoj0}
         &\sum_{j=0}^{j_{0}-1} \sum_{k \in \nu_{J,j}^{K}} |2^{-j H_{j,k}}D_{J,j}^{K,k}|^{2} \leq \sum_{j=0}^{j_{0}-1} \sum_{k=0}^{2^{j}-1} |2^{-j H_{j,k}}D_{J,j}^{K,k}|^{2} \nonumber \\
         &\quad \leq C\sum_{j=0}^{j_{0}-1} \sum_{k=0}^{2^{j}-1} 2^{-2j H_{j,k}}\left(\frac{2^j}{2^J} \right)^{(2H_{j,k}+1)\wedge 2)} \leq C 2^{-J((2\underline{H}+1)\wedge 2)},
     \end{align}
     where the constant $C$ does not depend on $J$ and $K.$

  Now, by  selecting $j_1=J^2$ in (\ref{eq22}) we obtain that for all
$t \in [0, 1]$ it holds
   \begin{equation} \label{J2}
   \sum_{j=J^2+1}^{+\infty} \sum_{k=0}^{2^j-1} 2^{-2j H_{j,k}} \left|h^{[H_{j,k}]}(2^jt-k)\right|^2  \leq C   2^{-2J^2\underline{H}}.
   \end{equation}
  Notice that by (\ref{eq32})
   \begin{eqnarray}
   |D_{J,j}^{K,k}| \leq |h^{[H_{j,k}]} (2^{j}d_{J,K+1}-k)| +2|h^{[H_{j,k}]} (2^jd_{J+1,2K+1}-k)|+|h^{[H_{j,k}]} (2^{j}d_{J,K}-k)|\nonumber
   \end{eqnarray}
and the application of the estimate (\ref{J2}) results in
     \begin{equation}\label{upJ2}
         \sum_{j=J^2+1}^{+\infty} \sum_{k \in \nu_{J,j}^{K}}|2^{-jH_{j,k}}D_{J,j}^{K,k}|^2 \leq \sum_{j=J^2+1}^{+\infty} \sum_{k=0}^{2^j-1} |2^{-jH_{j,k}}D_{J,j}^{K,k}|^2 \leq C 2^{-2J^2\underline{H}},
     \end{equation}
   where the constant $C$ does not depend on $J$ and $K$.

It follows from ({\rm \ref{eq1}}), ({\rm \ref{eq2}}), ({\rm \ref{eq27}}) and ({\rm \ref{eq30}}) that for all $j \in \{ j_{0}, \dots, J^{2}\}$ and $k \in \nu_{J,j}^{K}:$
     \begin{align}
         j|H_{j,k}-H_{j}(t_{0})| &\leq C j(1+j)|d_{j,k}-t_{0}| \nonumber\\
         & \leq C (1+J)^4 \left( |d_{j,k}-d_{J,K}| + |d_{J,K}-t_{0}| \right)  \leq C,\label{eq48}
     \end{align}
          where the constant $C$ does not depend on $J$, $K$, $j$ and $k$.

         By ({\rm \ref{eq45}}) and ({\rm \ref{eq48}}) we obtain
     \begin{equation*}
         2^{-jH_{j,k}} \leq C2^{-jH_{j}(t_{0})} \leq C 2^{-j(H(t_{0})-\varepsilon)}
     \end{equation*}
     and, therefore,
      \begin{equation*}
         \sum_{j=j_{0}}^{J^{2}} \sum_{k \in \nu_{J,j}^{K}} |2^{-jH_{j,k}}D_{J,j}^{K,k}|^{2} \leq C \sum_{j=j_{0}}^{J^{2}} \sum_{k \in \nu_{J,j}^{K}} |2^{-j(H(t_{0})-\varepsilon)}D_{J,j}^{K,k}|^2 \leq C (M_{J}^{1}+M_{J}^{2}),
     \end{equation*}
     where
     \begin{equation*}
         M_{J}^{1} = \sum_{j=j_{0}}^{J} \sum_{k=0}^{2^j-1} |2^{-j(H(t_{0})-\varepsilon)}D_{J,j}^{K,k}|^2,
     \end{equation*}
     and
     \begin{equation*}
         M_{J}^{2} = \sum_{j=J+1}^{J^2} \sum_{k=0}^{2^j-1} |2^{-j(H(t_{0})-\varepsilon)}D_{J,j}^{K,k}|^2.
     \end{equation*}

To establish the upper bound for $M_{J}^{2},$ one can use (\ref{eq32}) and the same approach as in (\ref{eq22}), which results in
     \begin{equation}\label{MJ}
         M_{J}^{2} \leq C\sum_{j=J+1}^{J^2} 2^{-2j(H(t_{0})-\varepsilon)} \left(\sum_{p \in \mathbb{Z}} (2+|p|)^{2\overline{H}-3}\right) \leq C 2^{-2J(H(t_{0})-\varepsilon)}.
     \end{equation}

     For $j \in \{j_{0}, \dots,J\}$, let us define the following integer indices
     \begin{equation*}
     \check{k}^{+}(J,K,j) \coloneqq \max \left \{ k \in \{ 0,\dots,2^j -1 \}: \frac{k}{2^j} \leq \dfrac{K+1}{2^J} \right \},
    \end{equation*}
     and
     \begin{equation}\label{eq54}
         \check{k}^{-}(J,K,j) \coloneqq \max \left \{ k \in \{ 0,\dots,2^j -1 \}: \frac{k+2}{2^j} \leq \dfrac{K}{2^J} \right \}.
     \end{equation}
   When $k \leq \check{k}^{-}$, it follows from the definition of  $D_{J,j}^{K,k}$ and Lemmas~\ref{lemma1} and \ref{lemma2} that all $(\cdot)_+$ terms from Lemma~\ref{lemma1}, that appear in the corresponding expressions for $h^{[H{j,k}]}(\cdot),$ have non-negative arguments. Thus, the positive part functions can be omitted. Applying the mean value theorem four times,  one obtains
     \begin{align}\label{eq55}
         \left|D_{J,j}^{K,k}\right| & = \left |\sum_{0\leq l,m\leq 2} \dfrac{\beta_{l}\beta_{m}}{H_{j,k}+{1}/{2}} \left( 2^{j} d_{J+1,2K+l} -k-\frac{m}{2} \right)^{H_{j,k}+\frac{1}{2}}\right | \nonumber \\
         & \leq C \left(\frac{2^j}{2^J}\right)^2 \left ( 2^{j} d_{J,K} -k-1\right)^{H_{j,k}-\frac{7}{2}},
     \end{align}
  where $\beta_{l},$ $l=0,1,2,$ are the same as in the proof of Lemma~{\rm \ref{lemma4}}. The multiplier $\left({2^j}/{2^J}\right)^2$ appeared from applying the mean value theorem twice for $l$ on intervals of lengths less than $2^j/2^J,$ and twice on intervals with lengths not exceeding 1.

    Next, notice that for $k\leq \check{k}^{-}(J,K,j),$ it follows from ({\rm \ref{eq54}}), that the power base in the last multiplier of (\ref{eq55}) can be bounded as follows
    \begin{equation}\label{eq56}
        2^{j} \left(d_{J,K}-\dfrac{\check{k}^{-}(J,K,j)+2}{2^{j}}\right)+ \check{k}^{-}(J,K,j) +2-k-1  \geq \check{k}^{-}(J,K,j)-k+1 \geq 1.
    \end{equation}
    Therefore, ({\rm \ref{eq55}}) and ({\rm \ref{eq56}}) imply the estimate
    \[\sum_{k=0}^{\check{k}^{-}(J,K,j)} |D_{J,j}^{K,k}|^{2} \leq C \left(\frac{2^j}{2^J}\right)^{4} \cdot\sum_{k=0}^{\check{k}^{-}(J,K,j)} (\check{k}^{-}(J,K,j) -k+1)^{2\overline{H}-7} \leq C \sum_{m=1}^{+\infty} m^{2\overline{H}-7} \cdot \left(\frac{2^{j}}{2^{J}}\right)^{4}.
    \]
    Next, let represent $M_{J}^{1}$ as the sum $M_{J}^{1,-} +M_{J}^{1,+},$ where the summands are defined and estimated below. For the first summand, it holds
    \begin{align}
        M_{J}^{1,-} &\coloneqq \sum_{j=j_{0}}^{J} \sum_{k=0}^{\check{k}^{-}(J,K,j)} |2^{-j(H(t_{0})-\varepsilon)}D_{J,j}^{K,k}|^{2} \leq C \sum_{j=j_{0}}^{J} 2^{-2j(H(t_{0})-\varepsilon)} \cdot 2^{-4(J-j)} \nonumber \\
        & = C \sum_{j=0}^{J-j_{0}} 2^{-2(J-j)(H(t_{0})-\varepsilon)} \cdot 2^{-4j}  = C 2^{-2J(H(t_{0})-\varepsilon)} \sum_{j=0}^{J-j_{0}} 2^{-2j(2-H(t_{0})+\varepsilon)} \nonumber \\
        & \leq C \left ( \sum_{j=0}^{+\infty} 2^{-2j(2-H(t_{0})+\varepsilon)}\right) \cdot 2^{-2J(H(t_{0})-\varepsilon)}.\label{M1-}
        \end{align}
The second summand is given by
        \begin{equation}\label{eq59}
            M_{J}^{1,+} \coloneqq \sum_{j=j_{0}}^{J} \sum_{k=\check{k}^{-}(J,K,j)+1}^{\check{k}^{+}(J,K,j)} |2^{-j(H(t_{0})-\varepsilon)}D_{J,j}^{K,k}|^{2}.
        \end{equation}
        The corresponding set
       \[
       \left\{ k \in \{0,\dots,2^j-1\}: \check{k}^{-}(J,K,j)+1 \leq k \leq \check{k}^{+}(J,K,j) \right\} \]
        \begin{equation}\label{eq60}=      \left\{ k \in \{0,\dots,2^j-1\}: \frac{K}{2^J}-\frac{2}{2^j}<\frac{k}{2^j}\leq \frac{K+1}{2^J} \right\}
        \end{equation}
        consists of at most 3 integer numbers.
        Indeed, for all $j \in \{j_{0},\dots,J\}$, the range of bounds for $k$ is strictly less than 3, as
        \begin{equation*}
            2^j\left(\frac{K+1}{2^J}-\left(\frac{K}{2^J}-\frac{2}{2^j} \right)\right) = \frac{2^j}{2^J} + 2 \leq 3.
        \end{equation*}

     Notice that (\ref{eq60}) can written as
        \[
             d_{J,K}-\frac{2}{2^j} < \frac{k}{2^j} \leq d_{J,K} + \frac{1}{2^J}, \]
     and
            \[-2  < k-2^{j} d_{J,K} \leq \frac{2^j}{2^J} \leq 1. \]
            Therefore, it holds $ |2^{j}d_{J,K}-k|\leq 2,$ which implies that
            \[|2^{j} d_{J+1,2K+l}-k| \leq 3, \quad \text{when} \quad  l \in \{ 0,1,2\}.\]

        Then, by ({\rm \ref{eq59}}), ({\rm \ref{eq60}}) and Assumption~\ref{ass1}  one obtains
        \begin{align}
            M_{J}^{1,+} &\leq C \sum_{j=j_{0}}^{J} 2^{-2j(H(t_{0})-\varepsilon)} \left(\frac{2^j}{2^J}\right)^{(2H_{j,k}+1)\wedge 2} \nonumber\\
            & \leq C \sum_{j=j_{0}}^{J} 2^{-2j(H(t_{0})-\varepsilon)} \cdot 2^{-(J-j)((2\underline{H}+1)\wedge 2)} \nonumber\\
            & = C \sum_{j=0}^{J-j_{0}} 2^{-2(J-j)(H(t_{0})-\varepsilon)} \cdot 2^{-j((2\underline{H}+1)\wedge 2)} \nonumber\\
            & = C 2^{-2J(H(t_{0})-\varepsilon)} \sum_{j=0}^{J-j_{0}} 2^{-2j((\underline{H}+\frac{1}{2})\wedge 1 -H(t_{0})+\varepsilon)} \leq C 2^{-2J(H(t_{0})-\varepsilon)}.\label{Mj+}
        \end{align}
The combination of the upper bounds (\ref{uptoj0}),  (\ref{upJ2}), (\ref{MJ}), (\ref{M1-}) and (\ref{Mj+}) gives
\begin{equation*}
            M_{J}^K\leq C 2^{-J((2\underline{H}+1)\wedge 2)}+C 2^{-2J^2\underline{H}}+C 2^{-2J(H(t_{0})-\varepsilon)}\leq C 2^{-2J(H(t_{0})-\varepsilon)}.
        \end{equation*}
        which proves Lemma~\ref{lemma6}.
    \end{proof}

 To complete the proof of Proposition~\ref{prop1}, notice that by Lemma~{\rm \ref{lemma 9}} and Proposition {\rm 4} in ~{\rm\cite{daoudi1998construction}}, it is enough to check that, for each fixed $t_{0} \in (0,1)$ and $\varepsilon >0$, there exists a constant $C$, which only depends on $\varepsilon$ and $t_0$, such, that for all $J\in \mathbb{Z}_{+}$ and $K \in \{0, \dots , 2^J -1 \}$ satisfying (\ref{eq27}) it holds
    \begin{equation*}
        \sigma(\Delta_{J,K}) \leq C 2 ^{-J (H(t_0)-\varepsilon)}.
    \end{equation*}
By applying (\ref{M+R}) and Lemmas~\ref{lemma5} and~\ref{lemma6}, one obtains the above estimate, which finishes the proof.
    \end{proof}

\section{Upper bound for H\"older exponent}\label{sec4}

This section derives an upper bound for the  H\"older exponent $\alpha_X(t).$

\begin{proposition}\label{proposition 2}
There is a universal event of probability $1,$ denoted by $\Omega_{2}^{*},$ such that on $\Omega_{2}^{*}$,  it holds
\[\alpha_{X}(t)\leq H(t)\quad \mbox{for all}\quad t\in (0,1).\]
\end{proposition}

\begin{proof}[Proof of Proposition~{\rm\ref{proposition 2}}]
    The proof of Proposition~{\rm \ref{proposition 2}} is inspired by the approach in the paper~{\rm\cite{AYACHE2020}}.

For all $J \in \mathbb{N}$ and $K \in \{0,\dots,2^J-1\},$ we set
\begin{equation}\label{eq63}
    \Delta^1_{J,K} \coloneqq X(d_{J,K+1})-X(d_{J,K}) = \sum_{j=0}^{+\infty} \sum_{\{j,k:\,d_{j,k}\leq d_{J,K+1}\}} a_{j,k} (J,K) \varepsilon_{j,k},
\end{equation}
where
\begin{equation}\label{eq64}
 \begin{aligned}
    a_{j,k} (J,K) &\coloneqq \int_{0}^{d_{J,K+1}} \left( (d_{J,K+1}-s)^{H_{j,k}-\frac{1}{2}}- (d_{J,K}-s)_{+}^{H_{j,k}-\frac{1}{2}} \right) h_{jk}(s)ds\\
    &\  = 2^{-jH_{j,k}} \left( h^{[H_{j,k}]} \left( 2^j d_{J,K+1} -k \right) -h^{[H_{j,k}]} \left( 2^jd_{J,K}-k \right)\right).
    \end{aligned}
\end{equation}
In the following, we represent $\Delta^1_{J,K}$ as a sum by splitting the double sum in~(\ref{eq63})  into non-overlapping sets of terms and analyzing the asymptotics of each summand separately as $J$ increases.

Let $\{e_{J}:\, J\geq 3\}$ be a non-decreasing sequence of integers greater or equal to 3, such that
\[\lim_{J\to +\infty} e_{J} = +\infty \quad \mbox{and}\quad e_{J} < 2^J.\]
For each $J \geq 3$ let us denote
\begin{equation} \label{eq65}
    {\mathcal L}_J \coloneqq \left\{ L \in \mathbb{N} : Le_{J}< 2^J\right\}.
\end{equation}

For all $j \in \mathbb{Z}_{+}$ and $L \in {\mathcal L}_{J},$ let us define
\begin{equation*}\label{eq66}
    \widetilde{\Lambda}_j(J,L) \coloneqq \left\{ k \in \{ 0, \dots, 2^j -1 \} : d_{J,(L-1)e_J+1} \leq d_{j,k} < d_{J,Le_J+1} \right\}
\end{equation*}
\begin{equation}\label{eq67}
    \check{\Lambda}_j(J,L) \coloneqq \left\{ k \in \{ 0, \dots, 2^j -1 \right\} :  d_{j,k} < d_{J,(L-1)e_J+1} \}.
\end{equation}

Then, we can introduce
\begin{equation}\label{eq68}
    \widetilde{\Delta}^1_{J,Le_J} \coloneqq \sum_{j=0}^{+\infty} \sum_{k \in \widetilde{\Lambda}_j(J,L)} a_{j,k}(J,Le_J) \varepsilon_{j,k}
\end{equation}
\begin{equation}\label{eq69}
    \check{\Delta}^1_{J,Le_J} \coloneqq \sum_{j=0}^{+\infty} \sum_{k \in \check{\Lambda}_j(J,L)} a_{j,k}(J,Le_J) \varepsilon_{j,k}.
\end{equation}


Let us consider the sequence $\{e_J\}$ of the form
\begin{equation}\label{eJ}
e_J: = 6[2^{J\delta}], \quad     \delta \in (0,1),
\end{equation}
where $[\cdot]$ denotes the integer part.

\begin{lemma}\label{lemma 17}
If $\{e_J\}$ satisfies~{\rm(\ref{eJ})}, then, for each $t_0 \in (0,1),$ there exists a constant $c(t_0) >0$ and a constant $\gamma>0,$ such that for all integers $J\geq 3$ and  $L \in {\mathcal L}_J$ satisfying
\begin{equation}\label{eq70'}
\left| t_0 -d_{J,Le_J} \right| \leq 2{Je_J} 2^{-J},
\end{equation}
it holds
 \begin{equation*}\label{eq70''}
     \sigma\left( \check{\Delta}^1_{J,Le_J} \right) \leq c(t_0) 2^{-J\left( H(t_0) +{\delta}\gamma \right)}.
 \end{equation*}
\end{lemma}

\begin{proof}[Proof of Lemma~{\rm \ref{lemma 17}}]
It follows from (\ref{eq64}), (\ref{eq67}), (\ref{eq69}), and Lemmas~{\rm \ref{lemma1}} and~{\rm \ref{lemma2}} that, for all~$J\geq 3$ and $L \in {\mathcal L}_J,$
\begin{align*}
    \sigma^2\left( \check{\Delta}^1_{J,Le_J}\right) &= \sum_{j=0}^{+\infty} \sum_{k \in \check{\Lambda}_{j}(J,L)} \left| a_{j,k} (J,Le_J) \right|^2 \\
    &= \sum_{j=0}^{+\infty} \sum_{k \in \check{\Lambda}_{j}(J,L)} 2^{-2jH_{j,k}} \left| h^{[H_{j,k}]} \left( 2^j d_{J,Le_J+1} -k \right) -h^{[H_{j,k}]} \left( 2^jd_{J,Le_J}-k \right) \right|^2.
\end{align*}

For each $j \in \mathbb{Z_{+}},$ let us define
\begin{equation}\label{eq71}
\bar{k}_j(J,L) \coloneqq \max \left\{ k \in \{0,\dots,2^j-1\}: d_{j,k}< d_{J,(L-1)e_J+1} \right\}.
\end{equation}

For each integer $k \leq \bar{k}_j(J,L)-2,$ the function $h^{[H_{j,k}]} (\cdot)$ belongs to the space $C^{\infty}([2^jd_{J,Le_J}-k,2^jd_{J,Le_J+1}-k]).$
    Thus, by applying  the mean value theorem three times, one gets that
    \begin{eqnarray}\label{eq71'}
    \left| h^{[H_{j,k}]} (2^j d_{J,Le_J+1}-k) -h^{[H_{j,k}]}(2^j d_{J,Le_J}-k) \right|^2 \nonumber \\ \leq C_1 2^{2(j-J)} \left( 1+(2^jd_{J,Le_J}-k-1) \right)^{2\overline{H}-5},
    \end{eqnarray}
where the constant $C_1$ does not depend on $J, L,j,$ and $k.$

The multiplier $2^{2(j-J)}$ in (\ref{eq71'}), appeared from the application of the mean value theorem: once on an interval of the length $2^{j-J},$ and twice on intervals, which lengths do not exceed 1. The power index $H_{j,k}$ was replaced by $\overline{H}$ without altering the inequality, as by (\ref{eq71})  it holds \[2^jd_{J, Le_J}-(k+1)>2^jd_{J,(L-1)e_J+1}-1-(k+1)>0.\]

    For each $j \in \mathbb{Z_{+}},$ let us define
    \begin{equation} \label{eq72}
        \check{\Lambda}_{j}^{2} (J,L) \coloneqq \left\{ k \in \check{\Lambda}_j(J,L) : d_{j,k+1}+ \dfrac{4}{(1+j)^2} < d_{J,(L-1)e_J+1}\right\}.
    \end{equation}
   Note that if $k \in \check{\Lambda}_{j}^{2} (J,L),$ then, since ${4}/{(1+j)^2}>2^{-j},$  it follows that  $k \leq \bar{k}_j(J,L)-2$  for all $j \in \mathbb{Z}_+.$

    By (\ref{eq71'}), (\ref{eq72}) and $d_{J,Le_J}-d_{J,(L-1)e_J+1}=(e_J-1)/{2^J},$  one obtains
   \begin{align*}
        \sum_{k \in \check{\Lambda}_{j}^{2} (J,L)}& \left| a_{j,k} (J,Le_J) \right|^2
        \leq C_1  \frac{2^{2j(1-\underline{H})}}{2^{2J}} \sum_{m=0}^{+\infty} \left( 1+ 2^j \left(\dfrac{e_J-1}{2^J}\right)+ \dfrac{4\cdot 2^j}{(1+j)^2} +m \right)^{2\overline{H}-5} \\
        &\leq C_1 2^{-2J} 2^{2j(1-\underline{H})} \int_{0}^{+\infty} \left( \dfrac{4\cdot 2^j}{(1+j)^2} +x \right)^{2\overline{H}-5} dx \\
        &\leq C_2 2^{-2J} 2^{2j(1-\underline{H})} \left( \dfrac{2^j}{(1+j)^2} \right)^{2\overline{H}-4} \leq C_2 2^{-2J} 2^{-2j(1+\underline{H}-\overline{H})} (1+j)^{8-4\overline{H}}.
    \end{align*}
Thus, as $\overline{H}-\underline{H}<1,$ it holds that
    \begin{equation}\label{eq73}
        \sum_{j=0}^{+\infty} \sum_{k \in \check{\Lambda}_{j}^{2} (J,L)} \left| a_{j,k} (J,Le_J)\right|^2 \leq C_2 2^{-2J} \sum_{j=0}^{+\infty} 2^{-2j(1+\underline{H}-\overline{H})} (1+j)^{8-4\overline{H}} \leq C_3 2^{-2J}.
\end{equation}
Thus, let us select $\gamma_1>0$ such that $H(t_0) +{\delta}\gamma_1\le 1.$

By estimating the difference of $h^{[H_{j,k}]} (\cdot)$ as in~(\ref{incr}) and applying~(\ref{eq64}), we deduce that for every $j_0 \in \mathbb{Z}_{+}$  and for any fixed $\alpha\in(0,1),$ there is a constant $C_4$ that does not depend on $j_0,$ such that
\begin{equation}\label{eq75}
    \sum_{j=j_0+1}^{J-[\log_{2}(e_J^\alpha)]} \sum_{k=\bar{k}_j(J,L)-1}^{\bar{k}_j(J,L)} \left| a_{j,k}(J,Le_J) \right|^2 \leq C_4 \sum_{j=j_0+1}^{J-[\log_2(e_J^\alpha)]} 2^{-2jH_{j,k}} 2^{2\left( (H_{j,k}+{1}/{2})  \wedge 1\right)(j-J)}.
\end{equation}

The application of (\ref{eq1}), (\ref{eq2}), (\ref{eJ}), (\ref{eq70'}) and (\ref{eq71})  for $j \leq J$ imply that, when $k$ is equal to $
\bar{k}_j(J,L)-1$ or $\bar{k}_j(J,L),$ it holds
\begin{align}
    j\left| H_{j,k} -H_j(t_0) \right| &\leq C_5 (1+j)^2 |t_0 -d_{j,k}|\leq C_5 (1+j)^2 \Big( |t_0-d_{J,Le_J}|  \nonumber\\
    & + |d_{J,Le_J} -d_{J,(L-1)e_J+1}| + |d_{J,(L-1)e_J+1}-d_{j,k}|\Big) \nonumber \\
    & \leq C_5 (1+j)^2 \left( Je_J2^{-J} + e_J 2^{-J} + 2^{-j+1}\right) \label{eq75'}\\
    & \leq C_6 (1+j)^2 2^{-j(1-2\delta)}  \leq C_7 \nonumber.
\end{align}
Therefore, by using (\ref{eq23'}) and (\ref{eJ}), the last sum in (\ref{eq75}) can be bounded by
\begin{align}
     & C_8 \sum_{j=j_0+1}^{J-[log_2(e_J^\alpha)]} 2^{-2j(H(t_0)-\varepsilon)} \cdot 2^{2((\underline{H}+{1}/{2})\wedge 1)(j-J)} \leq C_9 \frac {2^{2J\left( (\underline{H}+{1}/{2}) \wedge 1 -(H(t_0)-\varepsilon) \right)}}{2^{2J\left( (\underline{H}+\frac{1}{2})\wedge 1 \right)}}\nonumber \\
    & \quad \quad \times e_J^{2\alpha (H(t_0)-\varepsilon-(\underline{H}+\frac{1}{2})\wedge 1 )} \leq C_{10} 2^{-2J(H(t_0)-\varepsilon)} \cdot 2^{2J\alpha \delta (H(t_0)-\varepsilon-(\underline{H}+\frac{1}{2})\wedge 1 )}.\label{secest}
\end{align}
By (\ref{eq23'}) one can  select such $\gamma_2:=\alpha ((\underline{H}+\frac{1}{2})\wedge 1-H(t_0)+\varepsilon)-\varepsilon/\delta$  that it is positive for sufficiently small $\varepsilon.$

Now let us consider the set
\begin{equation*}
    \check{\Lambda}_{j}^{3}(J,L) \coloneqq \left\{ k \in  \check{\Lambda}_{j}(J,L) : k \leq \bar{k}_j (J,L)-2\ \mbox{and}\ d_{j,k+1} + \dfrac{4}{(1+j)^2} \geq d_{J,(L-1)e_J+1}\right\}.
\end{equation*}

When $k \in \check{\Lambda}_j^{3} (J,L),$ by applying the inequality $2^{-j} < 3(1+j)^{-2}$ and using (\ref{eJ})  when $j\leq J,$ similarly to (\ref{eq75'}), it follows that
\begin{align}
   & j \left| H_{j,k} - H_j(t_0) \right|  \leq   \ C_5 (1+j)^2 \Big( |t_0 -d_{J,Le_J}| + |d_{J,Le_J}-d_{J,(L-1)e_J+1}|   + |d_{J,(L-1)e_J+1}\nonumber \\
      &   -d_{j,k+1}| + |d_{j,k+1}-d_{j,k}|\Big)\leq C_5' (1+j)^2 \left( (J+1)e_J 2^{-J} + 7(1+j)^{-2} \right) \leq C_7^{'}. \label{eq76'}
    \end{align}

By  (\ref{eq45}), (\ref{eq71'}) and  (\ref{eq76'}), one obtains
\begin{align}
    \sum_{j=j_0+1}^{J-[\log_2(e_J^{\alpha})]}& \sum_{k \in \check{\Lambda}_j^{3}(J,L)} \left| a_{j,k}(J,Le_J) \right|^2  \leq C_{11} 2^{-2J} \sum_{j=j_0+1}^{J-[\log_2(e_J^{\alpha})]}  2^{2j \left( 1-H(t_0) + \varepsilon\right)} \sum_{m=0}^{+\infty} (1+m)^{2\overline{H}-5} \nonumber \\
    & \leq C_{12} 2^{-2J} \cdot 2^{2\left( J-\log_2(e_J^{\alpha}) \right)\left( 1+\varepsilon -H(t_0) \right)} = C_{12} 2^{-2J(H(t_0)-\varepsilon)} \cdot e_J^{-2\alpha(1+\varepsilon-H(t_0))} \nonumber \\
    & \leq C_{13} 2^{-2J(H(t_0)-\varepsilon)} \cdot 2^{-2J\alpha\delta(1+\varepsilon-H(t_0))} \label{eq78}.
\end{align}

Let us select such $\gamma_3:=\alpha (1+\varepsilon-H(t_0))-\varepsilon/\delta$  that it is positive for sufficiently small $\varepsilon.$

Let us now focus on the case where
$    J- [log_2(e_J^{\alpha})] < j \leq J^2.
$
For any such $j,$ one has
\begin{equation}\label{eq79}
    d_{J,Le_J}-d_{j,\bar{k}_j(J,L)+1} \geq \frac{1}{6} \cdot \frac{e_J}{2^J}.
\end{equation}

Indeed, it follows from (\ref{eq71}) that
\begin{equation*}
    0 \leq d_{j,\bar{k}_j(J,L)+1} -d_{J,(L-1)e_J+1} < 2^{-j}.
\end{equation*}

Thus, since $e_J \geq 3,$
\begin{align*}
   d_{J,Le_J} - d_{j,\bar{k}_j(J,L)+1} &= d_{J,Le_J} -d_{J,(L-1)e_J+1} +d_{J,(L-1)e_J+1} - d_{j,\bar{k}_j(J,L)+1} \\
   & > \dfrac{e_J-1}{2^J}-\frac{1}{2^j} \geq \dfrac{e_J-1}{2^J} - \dfrac{1}{2^{{J+1}-[log_2(e_J^\alpha)]}}  \geq \dfrac{e_J-1}{2^J}-\dfrac{e_J^\alpha}{2^{J+1}} \\
   & = \dfrac{e_J-2}{2^{J+1}} \geq \dfrac{e_J-2{e_J}/{3}}{2^{J+1}} = \frac{1}{6}\cdot \frac{e_J}{2^J}.
\end{align*}
 On the other hand, when $k \in \check{\underline{\Lambda}}_J^{3}(J,L) := \check{\Lambda}_J^{3}(J,L) \cup \{ \bar{k}_j(J,L)-1, \bar{k}_j(J,L)\},$ it follows from (\ref{eq75'}), (\ref{eq76'}) and $j\le J^2$, that
 \begin{equation}\label{eq80}
     j\left| H_{j,k}-H_j(t_0)\right| \leq C_5(1+j)^2 (J+1) e_J 2^{-J} + C_{14} \leq C_5 (1+J)^5 e_J 2^{-J} +C_{14} \leq C_{15}.
 \end{equation}

When $J-[log_2(e_J^\alpha)] \geq j_0,$ by using (\ref{eq9}),  (\ref{eq45}), (\ref{eq79}) and  (\ref{eq80}), one gets that
\begin{align}
    &\sum_{j=J-[log_2(e_J^\alpha)]+1}^{J^2} \sum_{k \in \check{\underline{\Lambda}}_J^{3}(J,L)} \left|a_{j,k} (J,Le_J) \right|^2 \leq C_{16} \sum_{j=J-[log_2(e_J^\alpha)]+1}^{J^2} 2^{-2j(H(t_0)-\varepsilon)} \nonumber\\
    &\times \sum_{k \in \check{\underline{\Lambda}}_J^{3}(J,L)} \left( \left| h^{[H_{j,k}]} (2^j d_{J,Le_J+1}-k)\right|^2 + \left| h^{[H_{j,k}]} (2^j d_{J,Le_J}-k) \right|^2 \right) \nonumber\\
    &\leq C_{17} \sum_{j=J-[log_2(e_J^\alpha)]+1}^{J^2} 2^{-2j(H(t_0)-\varepsilon)} \sum_{m=0}^{+\infty} \left( 2^j \frac{e_J}{6\cdot 2^J} +m \right)^{2\overline{H}-3} \nonumber\\
    &\leq C_{18} 2^{-2J(H(t_0)-\varepsilon)} \cdot e_J^{2\alpha (H(t_0)-\varepsilon)} \sum_{m=0}^{+\infty} \left( e_J^{1-\alpha}/3 +m \right)^{2\overline{H}-3} \nonumber \\
    & \leq C_{19} 2^{-2J(H(t_0)-\varepsilon)} 2^{2J\delta(\alpha (H(t_0)-\varepsilon)-(1-\overline{H})(1-\alpha))}\label{eq81}.
\end{align}

Note that the constant $C_{19}$ depends on $\alpha.$

To obtain a negative power in the second multiplier in (\ref{eq81}), the parameter $\alpha$ can be chosen such that $\alpha H(t_0)<(1-\alpha)(1-\overline{H}).$ For instance, it suffices to select $\alpha$ such that
$\alpha \overline{H}<(1-\alpha)(1-\overline{H}),$ which yields the condition $ \alpha < 1-\overline{H}.$

In this case we choose $\gamma_4:=\alpha (H(t_0)-\varepsilon)-(1-\overline{H})(1-\alpha)-\varepsilon/\delta$  that it is positive for sufficiently small $\varepsilon.$

Let us now consider the case of $j\geq J^2+1.$
By using (\ref{eq8}), (\ref{eq9}), (\ref{eq14}) and (\ref{eq67}) one obtains
\begin{align}
    \sum_{j=J^2+1}^{+\infty} \sum_{k \in \check{\Lambda}_j(J,L)} \left| a_{j,k}(J,L) \right|^2 &\leq C_{19} \sum_{j=J^2+1}^{+\infty} 2^{-2j\underline{H}} \sum_{m=0}^{+\infty} \left( 3+d_{J,Le_J} -d_{J,(L-1)e_J+1}+m\right)^{2\overline{H}-3}\nonumber \\
    &\hspace{-1cm}\leq C_{20} \sum_{j=J^2+1}^{+\infty} 2^{-2j\underline{H}} \sum_{m=0}^{+\infty}\left( 3+m\right)^{2\overline{H}-3} \leq C_{21} 2^{-2J^{2}\underline{H}}.\label{lastest}
   \end{align}
   In this case $\gamma_5>0$ can be selected to satisfy the inequality $H(t_0) +{\delta}\gamma_5\le J\underline{H}.$

By combining the upper bounds from equations (\ref{eq73}), (\ref{secest}), (\ref{eq78}), (\ref{eq81}), and (\ref{lastest}), and selecting the value of $\gamma:=\min_{i=1,...,5}{\gamma_i}>0,$  the proof of Lemma~\ref{lemma 17} is complete.
\end{proof}

\begin{lemma}\label{lem9}
    Let $t_0 \in (0,1)$ be arbitrary and $J\geq 3$ and $L \in {\mathcal L}_{J}$ be such that {\rm (\ref{eq70'})} holds. Then, the random variables $\widetilde{\Delta}^1_{J,Le_J},$ $L \in {\mathcal L}_J,$ are independent and
    \begin{equation*}\label{eq83}
        \sigma \left( \tilde\Delta^1_{J,Le_J} \right) \geq C\cdot 2^{-JH_{J-1}(t_0)},
    \end{equation*}
    where the constant $C>0$ does not depend on $J$ and $L.$
\end{lemma}

\begin{proof}[Proof of Lemma~{\rm\ref{lem9}}]
    The independence of the random variables $\widetilde{\Delta}^1_{J,Le_J},$ $L \in {\mathcal L}_J,$ follows from the independence of the random variables $\varepsilon_{j,k}$ and from the fact that the sets $\{ (j,k):\, j,k \in \mathbb{Z_{+}}, k \in   \widetilde{\Lambda}_j(J,L)\},$ $L \in {\mathcal L}_J,$ are disjoint.

    For all $J\geq 3$ and $L \in {\mathcal L}_J,$ $Le_J \in \widetilde{\Lambda}_J(J,L)$  as $d_{J,Le_J} < d_{J,Le_J+1}$ and $d_{J,Le_J-1}-d_{J,(L-1)e_J+1}=d_{J,e_J-2} >0$ because $e_J\geq 3.$

    Notice, that $d_{J-1,{Le_J}/{2}} = d_{J,Le_J}$ and by (\ref{eq4}) one gets \[h_{J-1,{Le_J}/{2}} = 2^{(J-1)/2}  \left({\mathbbm{1}}_{[{Le_J}/2^J,(Le_J+1)/2^J)}(s) -{\mathbbm{1}}_{[((Le_J+1)/2^J,(Le_J+2)/2^J   )}(s)\right).\] Therefore,
    it follows from (\ref{eq64}) and (\ref{eq68}) that for all $J\geq 3$ and $L \in {\mathcal L}_J$
   \begin{align}
        \sigma^2\left( \widetilde{\Delta}^1_{J,Le_J} \right) & = \sum_{j=0}^{+\infty}\sum_{k \in \widetilde{\Lambda}_j(J,L)} \left|a_{j,k}(J,Le_J)\right|^2  \geq \left|a_{J-1,{Le_J}/{2}} (J,Le_J) \right|^2\nonumber  \\
        & = \left| \int_{d_{J-1,{Le_J}/{2}}}^{d_{J-1,{Le_J}/{2}+1}} \left( d_{J,Le_J+1}-s \right)_{+}^{H_{J-1,{Le_J}/{2}}-{1}/{2}} h_{J-1,{Le_J}/{2}} (s) ds\right|^2\nonumber \\
        & = 2^{J-1} \left| \int_{d_{J,Le_J}}^{d_{J,Le_J+1}} \left( d_{J,Le_J+1}-s \right)^{H_{J-1,{Le_J}/{2}}-{1}/{2}}ds \right|^2\nonumber  \\
        & = \dfrac{2^{J-1}}{\left( H_{J-1,{Le_J}/{2}} +{1}/{2}\right)^2} \cdot 2^{-2J(H_{J-1,{Le_J}/{2}}+1/2) }  \geq \dfrac{2}{9} \cdot 2^{-2JH_{J-1,{Le_J}/{2}}}. \label{est1}
\end{align}

    It follows from (\ref{eq1}),  (\ref{eq2})  and (\ref{eq70'}) that there is a constant $C_2>0,$ which does not depend on $J$ and $L,$ such that for all $J\geq 3:$
    \begin{equation*}\label{eq84}
        J\left| H_{J-1,L{e_J}/{2}} -H_{J-1}(t_0) \right| \leq C_1 J^{2} \left| d_{J,Le_J}-t_0 \right| \leq C_1 J^3 e_J 2^{-J} \leq C_2.
    \end{equation*}
    Then, by applying the lower bound~{\rm(\ref{est1})}, one gets that
    \begin{align*}
         \sigma^2\left( \tilde\Delta^1_{J,Le_J} \right) & \geq C \cdot 2^{-2JH_{J-1}(t_0)-2J \left| H_{J-1,L{e_J}/{2}} -H_{J-1} (t_0) \right|} \\
         & \geq C \cdot 2^{-2C_2} 2^{-2JH_{J-1}(t_0)}.
    \\[-2.9\normalbaselineskip]\mathstrut
\end{align*}\qedhere\\
\end{proof}

It follows from the definition (\ref{eq65}), that for all $J \geq 3$ it holds
$    {\mathcal L}_J = \mathbb{N} \cap \left[1, {2^J}/{e_J}\right).$
For each $p \in \{1, \dots, [{2^J}/{Je_J}] -1\},$ let us define the next sets
\begin{equation}\label{eq85}
    {\mathcal L}_{J,p} := \mathbb{N} \cap \left[ (p-1)J,pJ \right],
\end{equation}
and
\begin{equation*}
    {\mathcal L}_{J,\left[ \frac{2^J}{Je_J} \right]} := \mathbb{N} \cap \left[ \left( \left[\frac{2^J}{Je_J}\right] -1\right)J, \frac{2^J}{e_J} \right].
\end{equation*}

By (\ref{eJ}), there exists an integer $J_0 \geq 3$ such that for all $J\geq J_0$ it holds $ [{2^J}/{Je_J}] \geq 2.$  For any $J\geq J_0$ and every $p \in \{ 1, \dots [{2^J}/{Je_J}] \}$ let us define
\begin{equation}\label{eq86}
    \mu_{J,p} \coloneqq \max_{L \in {\mathcal L}_{J,p}} \dfrac{|\tilde\Delta^1_{J,Le_J}|}{\sigma (\tilde\Delta^1_{J,Le_J})}.
\end{equation}

\begin{lemma}\label{lemma10}
There is a positive constant $C$ and an event of probability $1,$ denoted by~$\Omega_{2}^{**},$ such that on $\Omega_{2}^{**},$  it holds

\begin{equation*}
    \liminf_{J \rightarrow +\infty} \left\{ \inf_{1\leq p\leq [{2^J}/{Je_J}]} \mu_{J,p}\right\} \geq C.
\end{equation*}
    \end{lemma}

\begin{proof}[Proof of Lemma~{\rm\ref{lemma10}}]
 Observe, that by (\ref{eq85}), $\operatorname{card}({\mathcal L}_{J,p})  = J,$
where the notation $\operatorname{card}(\cdot)$ is used to denote the cardinality of a set.

Then, by using (\ref{eJ}), (\ref{eq86}) and the fact that $\dfrac{\tilde\Delta^1_{J,Le_J}}{\sigma (\tilde\Delta^1_{J,Le_J})},$ $L \in {\mathcal L}_J,$ are independent standard Gaussian random variables  for $J\geq 3,$ one obtains
 \[\sum_{J=3}^{+\infty} {\mathbb P} \left( \inf_{1\leq p\leq [{2^J}/{Je_J}]} \mu_{J,p}<C \right) \leq  \sum_{J=3}^{+\infty} \sum_{p=1}^{[{2^J}/{Je_J}]} {\mathbb P} \left( \max_{L \in {\mathcal L}_{J,p}} \dfrac{|\tilde\Delta^1_{J,Le_J}|}{\sigma (\tilde\Delta^1_{J,Le_J})}<C \right)
    \]
    \[=\sum_{J=3}^{+\infty} \sum_{p=1}^{[{2^J}/{Je_J}]} \prod_{L \in {\mathcal L}_{J,p}}{\mathbb P} \left( \dfrac{|\tilde\Delta^1_{J,Le_J}|}{\sigma (\tilde\Delta^1_{J,Le_J})}<C \right)=\sum_{J=3}^{+\infty} \left[\frac{2^J}{Je_J}\right] (\Phi(C))^{\operatorname{card}({\mathcal L}_{J,p})}\]
    \begin{equation}\label{sumC}\leq \sum_{J=3}^{+\infty} 2^{J(1-\delta+\log_2(\Phi(C))},\end{equation}
    where $Z\sim N(0,1)$ and $\Phi(C):={\mathbb P} \left( Z<C \right).$

    By choosing $C>0$ such that  $\Phi(C)<2^{-1+\delta}$ we conclude that the series~(\ref{sumC}) is convergent. Hence, the statement of the lemma follows directly from the Borel-Cantelli lemma.
\end{proof}

\begin{lemma}\label{lem11}
    For all $\omega \in \Omega_{2}^{**}$ and $t_0 \in (0,1),$ there is $J_1 (t_0,\omega) \geq J_0$ such that for all $J \geq J_1(t_0,\omega)$ there exists $L_J(t_0,\omega) \in {\mathcal L}_{J}$ for which {\rm(\ref{eq70'})} holds, i.e.
\begin{equation}\label{eq87}
    \left| t_0-d_{J,L_J(t_0,\omega)e_J} \right| \leq 2Je_J2^{-J}
\end{equation}
    and
    \begin{equation*}\label{eq88}
        \left| \tilde\Delta^1_{J,L_J(t_0,\omega)e_J} (\omega) \right| \geq C 2^{-JH_{J-1}(t_0)},
    \end{equation*}
    where $C$ is a positive constant that does not depend on $J$ and $L_J(t_0,\omega).$
\end{lemma}

\begin{proof}[Proof of Lemma~{\rm\ref{lem11}}]
    Notice, that for each $t_0 \in (0,1),$ there exists a non-random integer number $J_2(t_0),$ such that for all $J\geq J_2(t_0),$ it holds

    \begin{equation}\label{eq89}
        1 \leq p_{J}(t_0) \leq \left[ \frac{2^J}{Je_J}\right]-2 ,
    \end{equation}
    where $p_J(t_0) \coloneqq \left[ {2^Jt_0}/{Je_J} \right].$

     On the other hand, it follows from Lemma~{\rm\ref{lemma10}} that there is $J_{3} (\omega),$ $\omega\in\Omega_{2}^{**},$ such that for all $J \geq J_{3}(\omega):$
     \begin{equation}\label{eq91}
         \inf_{1\leq p\leq \left[{2^J}/{Je_J}\right]} \mu_{J,p} \geq C>0.
     \end{equation}

     Let us set $J_1(t_0,\omega) \coloneqq \max ( J_{2} (t_0), J_{3}(\omega) ).$

By (\ref{eq89}) and (\ref{eq91}), for all $J \geq J_{1} (t_0,\omega),$ it holds
\begin{equation}\label{eq92}
    \mu_{J,p_J(t_0)+1} \geq  C>0.
\end{equation}

     Then, (\ref{eq85}), (\ref{eq86}) and (\ref{eq92})  imply that there exists an integer

     \begin{equation}\label{eq93}
         L_J(t_0,\omega) \in \left[ p_J(t_0)J, (p_J(t_0)+1)J \right],
     \end{equation}
such that, by Lemma~\ref{lem9},
\begin{align*}
    \left| \tilde\Lambda^1_{J,L_J(t_0,\omega)e_J}  \right| &\geq C \sigma \left( \tilde\Delta^1_{J,L_J(t_0,\omega)e_J} \right)  \geq C 2^{-JH_{J-1}(t_0)}, \label{eq94}
\end{align*}
 which provides the second inequality in the statement of the lemma.

      Next, notice that by the choice of $p_J(t_0):$
      \begin{equation}\label{eq95}
          0 \leq t_0 -d_{J,p_J(t_0) Je_J} < Je_J 2^{-J},
      \end{equation}
      and (\ref{eq93}) implies that
            \begin{equation}\label{eq96}
          0 \leq d_{J,L_J(t_0,\omega) e_J} -d_{J,p_J(t_0) Je_J} \leq J e_J 2^{-J}.
      \end{equation}

      Combining the inequalities (\ref{eq95}) and (\ref{eq96}) one obtains (\ref{eq87}), which completes the proof of Lemma~\ref{lem11}.
\end{proof}

\begin{remark}\label{remark 22}
Using the notations from Lemma~{\rm\ref{lem11}}, since it holds that $H(t_0) = \liminf_{J \rightarrow +\infty} H_{J} (t_0),$ there exists a deterministic increasing sequence of positive integer numbers, $\{J_m\}_{m \in \mathbb{N}} \subset \mathbb{N},$ such that
for all sufficiently large $m \in \mathbb{N}$
\begin{equation*}
    \left| \tilde{\Delta}^1_{J_m, L_{J_m} (t_0,\omega)e_{J_m}} \right| \geq C 2^{-J_m(H(t_0)+{\delta}\gamma/{4})}.
\end{equation*}
\end{remark}
\begin{lemma}\label{lemma13}
There exists an event of probability $1,$ denoted by $\Omega_2^{***},$ such that for each $t_0 \in (0,1),$ there is a finite random variable $\zeta(t_0),$ so that on $\Omega_2^{***}$  it holds
    \begin{equation*}\label{eq97}
        \left| \check{\Delta}^1_{J,Le_J} \right| \leq \zeta(t_0) 2^{-J(H(t_0)+{3}\delta\gamma/2)}
    \end{equation*}
    for all integers $J \geq 3$ and $L \in {\mathcal L}_{J}$ satisfying {\rm(\ref{eq70'})}.
\end{lemma}

\begin{proof}[Proof of Lemma~{\rm\ref{lemma13}}]
Observe, that by (\ref{eJ}), there exists a constant $C>0$ such that \[\operatorname{card}({\mathcal L}_{J})  = \operatorname{card}\left(\left\{ L \in \mathbb{N} : L< 2^J/e_{J}\right\}\right)\le C\cdot 2^{J(1-\delta)}.\]

Then, by using the fact that $\dfrac{ \check{\Delta}^1_{J,Le_J} }{\sigma (\check{\Delta}^1_{J,Le_J})},$  $J\geq 3,$ $L \in {\mathcal L}_J,$ are standard Gaussian random variables and their tail probability estimate given by (\ref{normtail}) one obtains
 \[
        \sum_{J=3}^{+\infty} {\mathbb P} \left( \max_{L \in {\mathcal L}_{J}} \dfrac{\left| \check{\Delta}^1_{J,Le_J} \right|}{\sigma (\check{\Delta}^1_{J,Le_J})} \geq  \sqrt{2\ln(2)J}\right) \le  \sum_{J=3}^{+\infty} \sum_{L \in {\mathcal L}_{J}} {\mathbb P} \left( \dfrac{\left| \check{\Delta}^1_{J,Le_J} \right|}{\sigma (\check{\Delta}^1_{J,Le_J})} \geq  \sqrt{2\ln(2)J}\right)
    \]
    \[= \sum_{J=3}^{+\infty} \operatorname{card}({\mathcal L}_{J}) {\mathbb P} \left( |Z| \geq \sqrt{2\ln(2)J}\right)<   C\sum_{J=3}^{+\infty} 2^{J(1-\delta)}  \frac{\exp(-\ln(2)J)}{\sqrt{J}} \]
    \[\le C\sum_{J=3}^{+\infty} 2^{-J\delta} <+ \infty.\]
Hence, by the Borel-Cantelli lemma and Lemma~\ref{lemma 17}, the next inequality holds  on $\Omega_2^{***}$
    \begin{equation*}
        \left| \check{\Delta}^1_{J,Le_J} \right| \leq \zeta(t_0) 2^{-J(H(t_0)+\delta\gamma)}\sqrt{J}.
    \end{equation*}

Finally, the statement of Lemma~\ref{lemma13} is established, as there exists a positive constant $C$ such that
\begin{equation}\label{sqJ}
\sqrt{J}\leq C\cdot 2^{J\delta\gamma/2}
\end{equation}
for all $J\ge 3.$
\end{proof}

\begin{remark}
    Let the event $\Omega_2^{*}$ be defined as $\Omega_2^{*} := \Omega_2^{**} \cap \Omega_3^{***}.$ Then, its probability is~$1,$ and since for all $J \geq 3$ and $L \in {\mathcal L}_J$ it holds
    \begin{equation*}
        \Delta^1_{J,Le_J} = \tilde{\Delta}^1_{J,Le_J} + \check{\Delta}^1_{J,Le_J},
    \end{equation*}
by Remark~{\rm\ref{remark 22}}, Lemma~{\rm\ref{lemma13}} and the triangle inequality, for each $t_0 \in (0,1),$ one has on $\Omega_2^{*}$ that

 \begin{equation}\label{eq98}
     \limsup_{J \rightarrow +\infty} \left\{ 2^{J(H(t_0)+ \delta\gamma/{4})} \max_{L \in {\mathcal L}_J (t_0)} |\Delta^1_{J,Le_J}| \right\} \geq C >0,
 \end{equation}
where  \begin{equation}\label{eq99}
     {\mathcal L}_J(t_0) = \left\{ L \in {\mathcal L}_J: |t_0 - d_{J,Le_J}| \leq 2J e_J 2^{-J} \right\}.
 \end{equation}
\end{remark}

 To complete the proof of Proposition~\ref{proposition 2},  let us use Definition~\ref{def1} and show that on~$\Omega_2^{*},$ for each $t_0 \in (0,1),$

  \begin{equation*}\label{eq100}
      \alpha_{X} (t_0) \leq H(t_0) + 2 \delta.
  \end{equation*}

By contradiction, suppose that there exists $\bar{\omega} \in \Omega_2^{*}$ and $\bar{t}_0(\bar{\omega}) \in (0,1),$ such that
 \begin{equation*}
     \limsup_{h \rightarrow +\infty} \left\{ |h|^{-H(\bar{t}_0(\bar{\omega}))-2\delta} \left| X(\bar{t}_0 (\bar{\omega})+h) - X (\bar{t}_0 (\bar{\omega}))\right| \right\} =0.
 \end{equation*}

 Then, by applying the upper bound on $h=|\bar{t}_0(\bar{\omega}) - d_{J,Le_J}|,$ it follows from the equation (\ref{eq99}) that the following limit is also vanishing
 \begin{equation*}
     \limsup_{J \rightarrow +\infty} \left\{ \left( \dfrac{2^J}{Je_J}\right)^{H(\bar{t}_0 (\bar{\omega}))+2\delta} \max_{L \in {\mathcal L}_J(\bar{t}_0 (\bar{\omega}))} \left| \Delta^1_{J,Le_J} (\bar{\omega}) \right| \right\} =0,
 \end{equation*}
 and consequently, by~(\ref{eJ}) and~(\ref{sqJ}), that
 \begin{equation}\label{eq101}
     \limsup_{J \rightarrow +\infty} \left\{ 2^{J(1-\frac{5}{4}\delta) (H(\bar{t}_0(\bar{\omega}))+ 2\delta)} \max_{L \in {\mathcal L}_J(\bar{t}_0 (\bar{\omega}))} \left| \Delta^1_{J,Le_J} (\bar{\omega}) \right|\right\}  =0.
 \end{equation}

If $\delta<\frac{1}{2}\left(1-\frac{3\gamma}{4}\right)$ is selected, then, for any  $H(\bar{t}_0(\bar{\omega})\in[0,1],$ it holds
\begin{equation*}
    \left( 1-\delta\left(1+\frac{\gamma}{2} \right) \right)\left( H(\bar{t}_0(\bar{\omega}))+ 2\delta \right) > H(\bar{t}_0(\bar{\omega})) +\frac{\delta\gamma}{4}.
\end{equation*}
 and, therefore, (\ref{eq101}) contradicts ({\rm \ref{eq98}}). It completes the proof of Proposition~\ref{proposition 2}.
\end{proof}

Now, by combining Propositions~\ref{prop1} and~\ref{proposition 2}, we obtain the following result, which demonstrates that the GHBMP has the required H\"older exponent.
\begin{theorem}
There is a universal event of probability $1,$ denoted by $\Omega^{*},$ such that on~$\Omega^{*}$,  it holds
\[\alpha_{X}(t)= H(t)\quad \mbox{for all}\quad t\in (0,1).\]
\end{theorem}

\section{Simulation studies}\label{sec5}

This section presents simulation studies to validate the obtained theoretical results. The GHBMP was simulated using four different Hurst functions. Subsequently, the Hurst functions were estimated using the realizations and compared with the theoretical Hurst functions to confirm the accuracy of the model. Notice that some of the used theoretical Hurst functions did not satisfy Assumption~\ref{ass1}. However, simulation studies still demonstrated that the resulting process possessed the required properties. This suggests that the approach is applicable in more general cases.

To simulate the GHBMP the truncated version of the formula (\rm{\ref{eq6}}) was used, where the first summation was considered up to the level $J:$
\[
    X_J(t) \coloneqq \sum_{j=0}^{J}  \sum_{k=0}^{2^{j}-1}
     \left(\int_{0}^{1} (t-s)_{+}^{H_{j}(k/{2^j})-{1}/{2}} h_{j,k}(s)ds \right)\varepsilon_{j,k}.
\]

The following estimation method of the Hurst function given in publication \rm{\cite{ayache2022uniformly}}  was applied
\begin{equation*}
    \hat{H}_N^{Q} (I_{N,n}) \coloneqq \min \left\{ \max \left \{ \log_{Q^2} \left( \dfrac{V_{N}(I_{N,n})}{V_{QN}(I_{N,n})} \right) ,0 \right \} ,1 \right\},
\end{equation*}
where $\{I_{N,n}\}$ is a finite sequence of compact subintervals of $[0,1]$ and $Q \geq 2$ is a fixed integer. $V_{N}(I_{N,n})$ denotes the generalized quadratic variation of the multifractional process $X$ on $I_{N,n}$ that is defined as
$$ V_{N} (I_{N,n}) \coloneqq |\nu_{N} (I_{N,n})|^{-1 } \sum_{k \in \nu_{N} (I_{N,n})} |d_{N,k}|^{2},$$
where $|\nu_{N} (I)|$ denotes the cardinality of the set
$\nu_{N} (I) \coloneqq \left\{ k \in \{ 0,\dots, N-L\}:\right.$ $\left. k/N \in I\right\}$ and
$$d_{N,k} = \sum_{l=0}^{L} a_{l} X\left(\frac{k+l}{N}\right)$$
are the generalized increments of the multifractional process $X.$
Here, $0\leq k \leq N-L$, $N\ge L$, where $L\geq 2$ is an arbitrary fixed integer and $\{a_{l}\}$ are coefficients defined for $l \in \{0,\dots,L\}$ by
$$a_{l} \coloneqq (-1)^{L-l} \binom{L}{l} \coloneqq (-1)^{L-l} \dfrac{L!}{l!(L-l)!}.$$

In each of the following four cases, the multifractional processes were simulated using $J=20$ and the equally spaced grid of $(2^{18}+1)=262145$ time points in the time interval $[0,1]$. To estimate the Hurst function, $Q=2$ and $L=2$ were chosen and the partition $\{I_{N,n}\}$ of the time interval $[0,1]$ into $N=100$ subintervals was used.

First, the GHBMP was simulated for a constant Hurst function of $H(t)\equiv 0.5$ as depicted in Figure~\ref{fig1}. This case corresponds to the standard Brownian motion and the roughness of realisations of this process should not change on the interval $[0,1]$. Then, the processes were simulated for two Hurst functions, $H(t)=0.2+0.45t$ and $H(t)=0.5-0.4\sin(6\pi t).$ In all these three cases, $H_j(t) =H(t).$ In the second case, it is anticipated that the roughness of the process decreases. In the third case, the roughness is expected to exhibit oscillating behaviour. This aligns with the plots depicted in Figures~\ref{fig1}, \ref{fig2} and~\ref{fig3}, respectively.

The Hurst functions employed in the simulations are depicted in blue, while the estimated Hurst functions are represented in red. Due to the randomness and the small intervals used, the estimated Hurst function exhibits rapid fluctuations around the true one. Consequently, the LOESS method was employed to smooth the estimated Hurst function and reveal its general pattern, which is illustrated in green. According to the obtained plots, the true Hurst function and the smoothed Hurst function approximately coincide in each of the three considered cases. This close alignment between the theoretical and empirical Hurst functions validates the proposed methodology.

\begin{figure}[!hb]
\centering
\begin{subfigure}[c]{0.4\textwidth}
\centering
    \includegraphics[trim=0 0 4cm 0, clip, width=\textwidth, height= 1.2\textwidth]{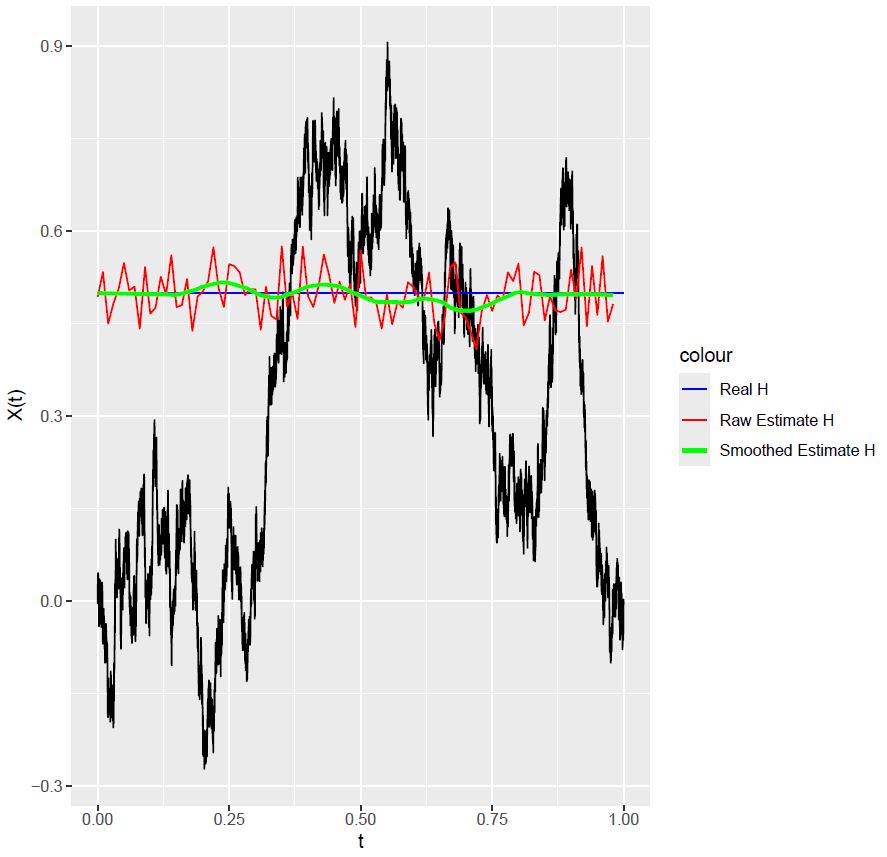}
    \caption{Case of $H\equiv 0.5$}
    \label{fig1}
\end{subfigure}\hspace{4mm}
\begin{subfigure}[c]{0.54\textwidth}
\centering
    \includegraphics[trim=0 0 0 5mm, clip,width=\textwidth, height= 0.885\textwidth]{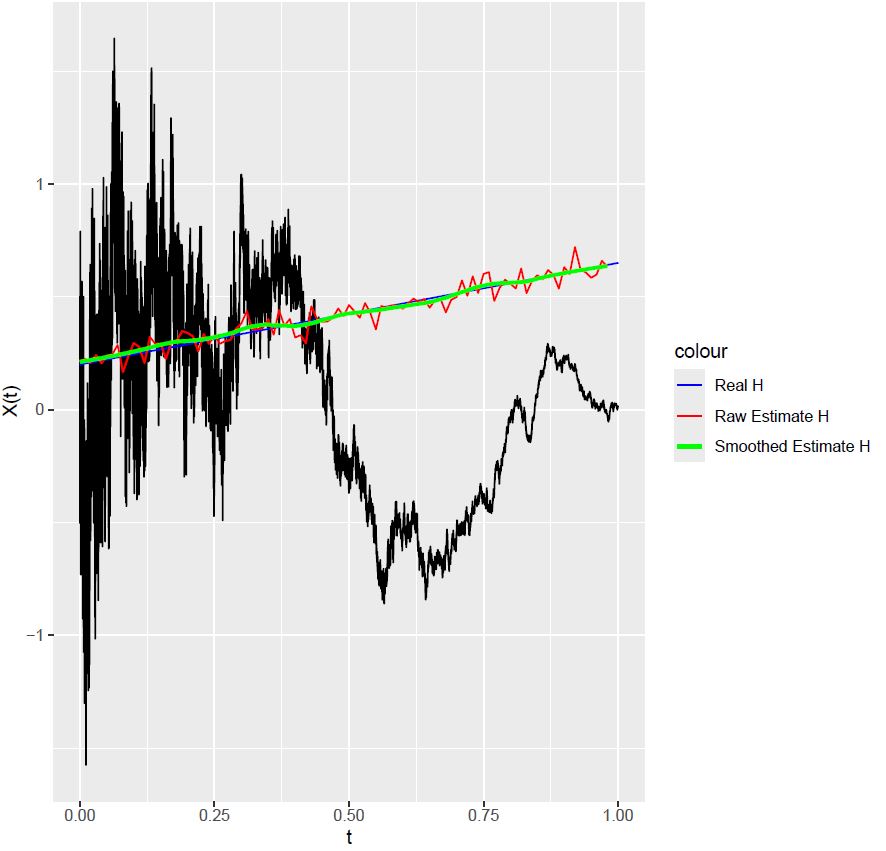}
    \caption{Case of  $H(t)=0.2+0.45t$}
    \label{fig2}
\end{subfigure}

\begin{subfigure}[c]{0.4\textwidth}
\centering
    \includegraphics[trim=0 0 4cm 0, clip, width=\textwidth, height= 1.21\textwidth]{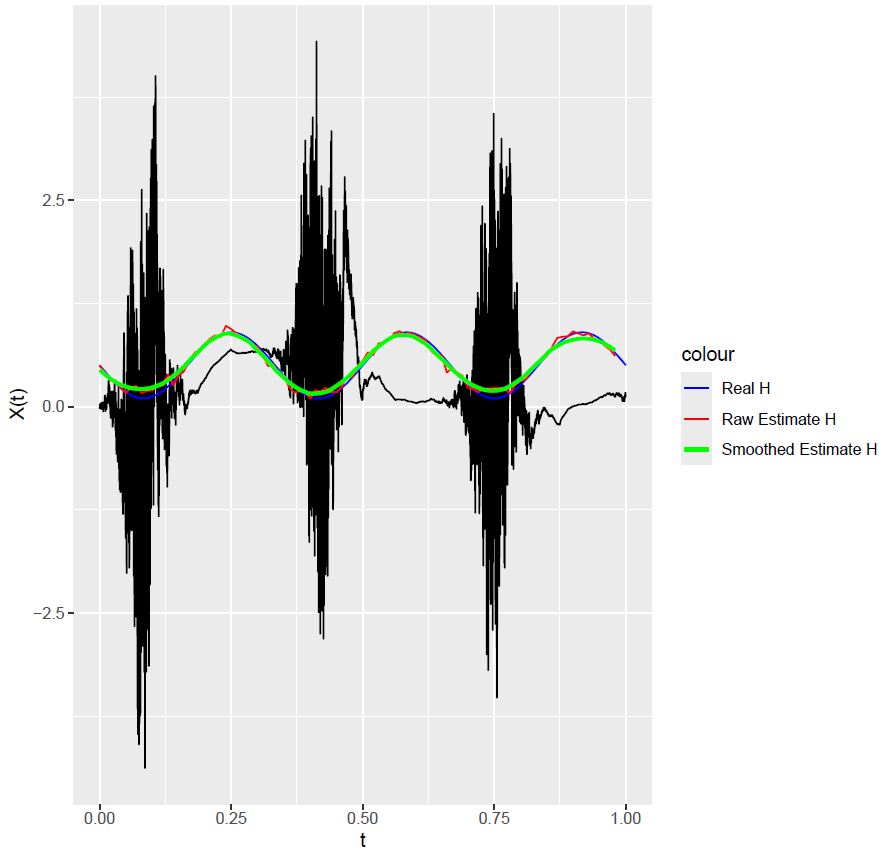}
    \caption{Case of $H(t)=0.5-0.4\sin(6\pi t)$}
    \label{fig3}
\end{subfigure}\hspace{4mm}
\begin{subfigure}[c]{0.54\textwidth}
\centering
    \includegraphics[trim=0 0 0 5mm, clip,width=\textwidth, height= 0.885\textwidth]{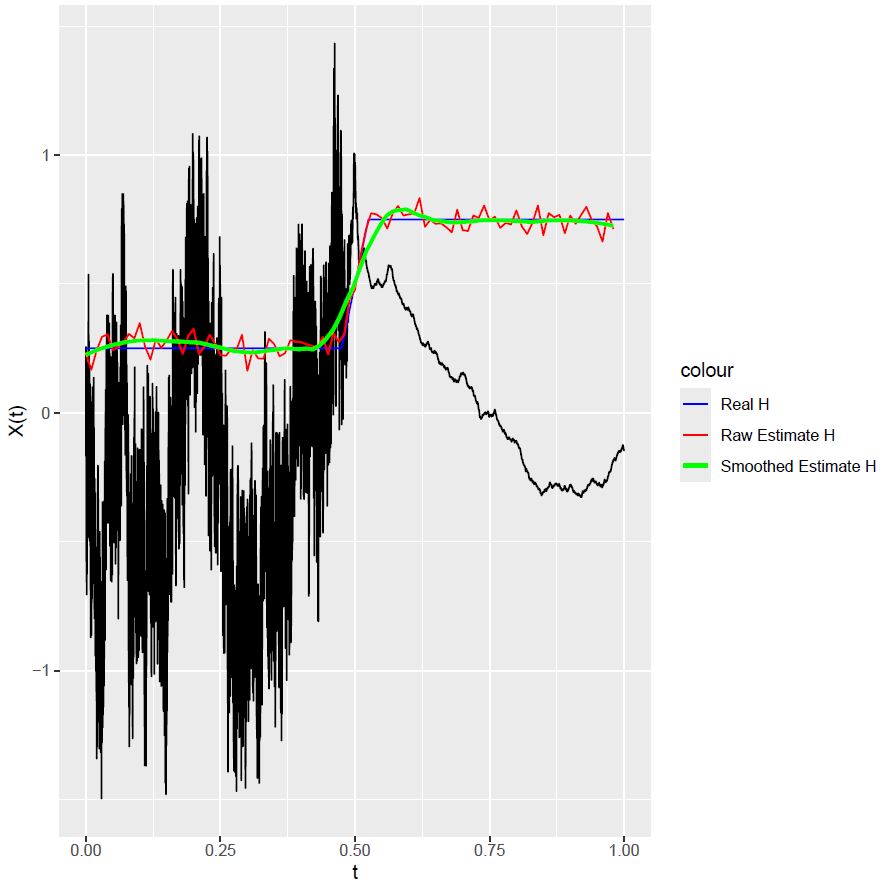}
    \caption{Case of $H(t)$ given by ({\rm \ref{eq102}})}
    \label{fig4}
\end{subfigure}
\caption{Realisations of GHBMP and corresponding Hurst functions}
\end{figure}

Furthermore, GHBMP was simulated using the following sequence of functions
\begin{equation*}
    H_{j}(t) =
    \begin{cases}
        {1}/{4}, & \text{if } t \in \left[0,{1}/{2}-{1}/({2j})\right]; \\
        {jt}/{2} + \left({1}/{2} -{j}/{4} \right), & \text{if } t \in \left[{1}/{2}-{1}/({2j}),{1}/{2}+{1}/({2j})\right]; \\
        {3}/{4}, & \text{if } t \in \left[{1}/{2}+{1}/({2j}),1\right],
    \end{cases}
\end{equation*}
that converges to the Hurst function with a discontinuity
\begin{equation}\label{eq102}
    \lim_{j \to \infty} H_{j}(t) = H(t) =
     \begin{cases}
        {1}/{4}, & \text{if } t \in \left[0,{1}/{2}\right]; \\
        {3}/{4}, & \text{if } t \in \left[{1}/{2},1\right].
    \end{cases}
\end{equation}

The simulation for this Hurst function was performed for $J=20$. Likewise in the previous cases, the Hurst function was estimated and the smoothed estimate was obtained and is depicted in Figure~\ref{fig4}. When $J$ increases, the smoothed Hurst function can be expected to approach the limit~(\ref{eq102}).  As in the previous three instances, the actual and smoothed Hurst functions exhibit a very close proximity in this case.

To analyze the deviations of the estimates from the theoretical functions used in simulations, we computed absolute differences between the theoretical and estimated Hurst functions across 100 sub-intervals for the Hurst function $H(t)=0.5-0.4\sin(6\pi t)$. This was repeated for 30 realizations for different $J$ and time intervals. The time intervals were defined by sequences from $0$ to $1$ with the step ${1}/{2}^n,$ where $n$ increased with $J$  for accurate wavelet transform computations. Figure~\ref{fig6} depicts the boxplots of the absolute differences and Figure~\ref{fig7} visualizes the maximum absolute differences of the total 30 realizations for $J$ from 14 to 20 and $n$ from 10 to 15. Further, the realizations were averaged, and the boxplots were obtained for the different $n$ and~$J$ parameters (see Figure~\ref{fig5}). Moreover, the average, maximum, and mean squared absolute difference for these averaged values are presented in Table~\ref{table:1}.
\begin{figure}[htb!]
    \centering
\begin{subfigure}[c]{0.32\textwidth}
    \includegraphics[trim=0mm 6mm 0 0, clip, width=\textwidth, height= 1.4\textwidth]{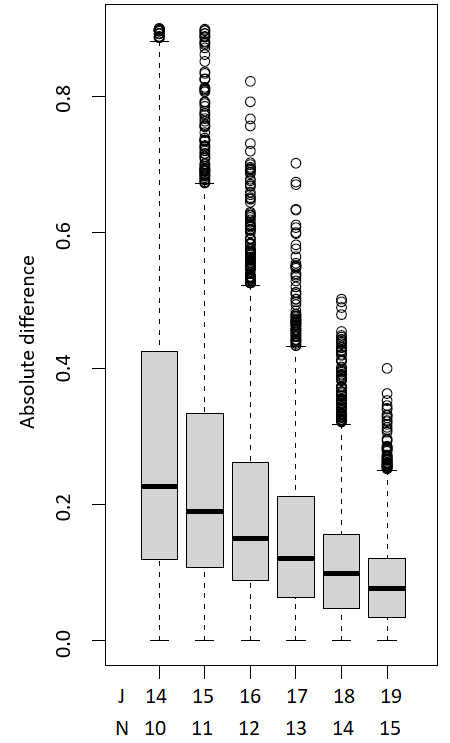}
    \caption{total  differences}
    \label{fig6}
\end{subfigure}
\begin{subfigure}[c]{0.32\textwidth}
    \includegraphics[trim=3.9mm 7mm 0 0, clip, width=\textwidth, height= 1.4\textwidth]{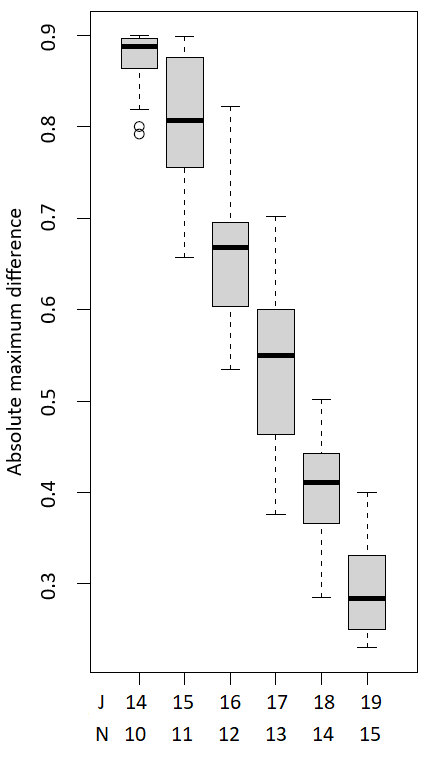}
    \caption{maximum differences}
    \label{fig7}
\end{subfigure}
\begin{subfigure}[c]{0.32\textwidth}
    \includegraphics[trim=11mm 10mm 0 0, clip, width=\textwidth, height= 1.4\textwidth]{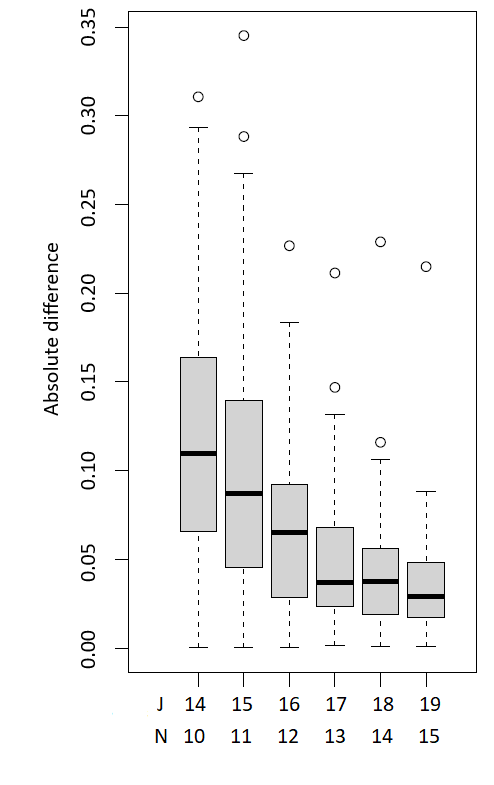}
    \caption{averaged differences}
    \label{fig5}
\end{subfigure}
\caption{Boxplots of differences of estimates and $H(t)=0.5-0.4\sin(6\pi t)$}
\end{figure}
\begin{table}[!htb]
\begin{center}
    \begin{tabular}{c c c c c}
 \hline
$J$ & $n$ & Average & Maximum & Mean squared \\
  & &difference & difference & difference \\
 \hline
 14 & 10 & 0.1206 & 0.3107 & 0.0198 \\
 \hline
 15 & 11 & 0.0979 & 0.3452 & 0.0143 \\
 \hline
 16  & 12 & 0.0664 & 0.2267 & 0.0064 \\
 \hline
 17 & 13 & 0.0473 & 0.2114 & 0.0035 \\
 \hline
 18 & 14 & 0.0411 & 0.2289 & 0.0027 \\
 \hline
 19 & 15 & 0.0347 & 0.2149 & 0.0020 \\
 \hline
\end{tabular}
\end{center}
\caption{Differences between estimates and $H(t)=0.5-0.4\sin(6\pi t)$ for increasing $J$}
\label{table:1}
\end{table}

The numeric values presented in the table and plots confirm that the accuracy of approximations of $H(t)$ improves with increases in both $J$ and the temporal resolution. Simulation studies suggest that in practical applications, utilising values of $J$ within the range of 15-17 will yield highly accurate modelling of desirable Hurst functions.

\section{Conclusion and future studies}\label{sec6}

Gaussian Haar-based multifractional processes were constructed using the Haar wavelet approach to model data with multifractal behaviour. The Haar wavelets were used due to their simplicity and computational efficiency, leading to high-speed computations. The suggested class of GHBMP can efficiently model sharp changes in the roughness of processes. Its theoretical properties were analyzed. Simulation studies were conducted for several Hurst functions. Based on the simulations and estimations performed, it can be concluded that the process possesses the required properties. Furthermore, the simulations confirm that the proposed approach works effectively, even in cases where the Hurst function exhibits discontinuities

To advance the proposed approach, future studies can be focused on:
\begin{itemize}
    \item developing similar models for processes with random Hurst functions;
    \item modifying the approach for the multidimensional case of random fields;
    \item converting the developed code into an R package, making it accessible to other researchers.
\end{itemize}

\section*{Acknowledgments}

Funding: This research was supported under the Australian Research Council's Discovery Projects funding scheme (project number  DP220101680). A.Olenko is grateful to Laboratoire d'Excellence, Centre Europ\'{e}en pour les Math\'{e}matiques, la Physique et leurs interactions (CEMPI, ANR-11-LABX-0007-01), Laboratoire de Math\'{e}matiques Paul Painlev\'{e}, France, for support and providing him with the opportunity to conduct research at the Universit\'{e} de Lille for a month. He was also partially supported by La Trobe University's SCEMS CaRE and Beyond grant. A.Ayache expresses his gratitude to La Trobe University for its support and for the opportunity to conduct joint research during his four-week visit in 2023.

Computations in this research were done using the Linux computational cluster Gadi of the National Computational Infrastructure (NCI), Australia.

\section*{Data availability}
The simulation studies were performed using the software R version 4.3.1. The R code is freely available in the folder ``Research materials'' from the website \url{https://sites.google.com/site/olenkoandriy/}.

\bibliographystyle{abbrv}
\bibliography{references}

\end{document}